\documentclass[11pt,a4paper,english]{article}

\usepackage{amssymb,amscd, amsmath, babel, amstext, amsthm, latexsym}
\usepackage{color}

\newtheorem{theorem}{Theorem}
\newtheorem{lemma}[theorem]{Lemma}
\newtheorem{corollary}[theorem]{Corollary}
\newtheorem{proposition}[theorem]{Proposition}

\newtheorem{remark}[theorem]{Remark}
\newtheorem{definition}[theorem]{Definition}

\numberwithin{theorem}{section}
\numberwithin{equation}{section}

\newtheorem{lem}[theorem]{Lemma}
\newtheorem{defn}[theorem]{Definition}

\newcommand{\mcal}{\mathcal}

\newcommand{\R}{\mathbb{R}}

\newcommand{\F}{\ensuremath{\mathcal{F}}}
\newcommand{\Ff}{\ensuremath{\mathbb{F}}}

\newcommand{\B}{\ensuremath{\mathcal{B}}}

\newcommand{\Pp}{{\mathcal{P}}_{st} }
\newcommand{\wPp}{{\widetilde{\mathcal{P}}}_{st}}

\newcommand{\esssup}{\ensuremath{{ \rm esssup} }}
\newcommand{\essmax}{\ensuremath{{\rm essmax} }}
\newcommand{\essinf}{\ensuremath{{\rm essinf} }}

\def\ben{\begin{enumerate}}
\def\een{\end{enumerate}}
\def\bit{\begin{itemize}}
\def\eit{\end{itemize}}

\title{Fully-dynamic risk-indifference pricing \\and no-good-deal bounds}

\author{Jocelyne Bion-Nadal\thanks{UMR 7641 CNRS - Ecole Polytechnique. Ecole Polytechnique, 91128 Palaiseau Cedex, France. Email: jocelyne.bion-nadal@cmap.polytechnique.fr}\: and Giulia Di Nunno\thanks{Department of Mathematics,
University of Oslo, P.O. Box 1053 Blindern, N-0316 Oslo Norway.
Email: giulian@math.uio.no}
\thanks{Norwegian School of Economics (NHH), Helleveien 30, N-5045 Bergen, Norway.}}
\date{April $15^{th}$, 2019}

\begin{document}
\maketitle

\begin{abstract}
The seller's risk-indifference price evaluation is studied. We propose a dynamic risk-indifference pricing criteria derived from a fully-dynamic family of risk measures on the $L_p$-spaces for $p\in [1,\infty]$. 
The concept of fully-dynamic risk measures extends the one of dynamic risk measures by adding the actual possibility of changing the risk perspectives over time. The family is then characterised by a double time index. Our framework fits well the study of both short and long term investments.
In this dynamic framework we analyse whether the risk-indifference pricing criterion actually provides a proper convex price system.
It turns out that the analysis is quite delicate and necessitates an adequate setting and the extension of the price operators and underlying risk measures. 

Furthermore, we consider the relationship of the fully-dynamic risk-indifference price with no-good-deal bounds. Recall that no-good-deal pricing guarantees that not only arbitrage opportunities are excluded, but also all deals that are ``too good to be true''.
We shall provide necessary and sufficient conditions on the fully-dynamic risk measures so that the corresponding risk-indifference prices satisfy the no-good-deal bounds.
We remark that the use of no-good-deal bounds also provides a method to select the risk measures and then construct a proper fully-dynamic risk-indifference price system in the $L_2$-spaces.

\vspace{2mm}
{\it Key-words:} convex prices, risk-indifference prices, time-consistency, extension theorems, dynamic risk measures, no-good-deal.
\end{abstract}

\section{Introduction}
Risk-indifference pricing was proposed as an alternative to utility indifference, as a pricing approach in incomplete markets, see e.g. \cite{Xu2006}.
The idea is to replace the performance of a portfolio described in terms of the agent's utility function with the performance based on the risk exposure measured by convex risk measures.
Indeed, instead of maximising the utility, which is often quite difficult for a trader to write down explicitly, the risk-indifference criterion minimises the risk exposure.
The point of connection between the two approaches is represented by the entropic risk measure which corresponds to an exponential-type utility.
The price of a financial claim is then defined as the amount such that the risk exposure remains the same whether the optimal trading strategy includes the claim or not.

This pricing approach was introduced on a fixed time horizon $T$ characterised by the expiration time of the claim $X$. The corresponding price $x_{0T}(X)$ itself is evaluated at the present time $s=0$.
A large part of the literature has focused on the stochastic control techniques necessary to compute the price $x_{0T}(X)$ of $X$, typically a financial derivative. 
The different  techniques depend on the choice of risk measure, on the market dynamics of the fundamental traded underlying assets, and on the information flow available.
See e.g. \cite{ES2011, GH2007, OS2009}.
Again depending on the underlying price dynamics and the information flow available, one can extend some of these results to obtain the price evaluation for any time $s\in [0,T]$.

\vspace{2mm}
On the other side, \cite{KS} presents a systematic study of the risk-indifference price system $(x_{sT}(X))_{s}$ for essentially bounded claims $X\in L_\infty(\F_T)$.
Similar to the present paper, also \cite{KS} is kept free of the particular choice of underlying price dynamics.

In \cite{KS} the evaluation of each price $x_{sT}(X)$ is linked with one element $\rho_s$ from the dynamic convex risk measure $(\rho_s)_s$. The risk measure $\rho_s$ evaluates at time $s$ the risks of financial positions at $T$.

However, particularly thinking of long time horizons, it seems reasonable to allow for the possibility of modelling changes in the risk evaluation criteria. 
For this reason, in this paper, we consider a \emph{fully-dynamic} convex risk measure, namely a family of risk measures indexed by {\it two} points in time $(\rho_{st})_{s,t}$. 
Working with fully-dynamic risk measures has important consequences both from the point of view of the modelling and the one of the mathematical analysis.

To illustrate the modelling perspective, we can consider two times $s_1, s_2$ such that $0 < s_1 < s_2 <T$ and a $\mathcal{F}_{s_2}$-measurable claim $X$.\\
The risk evaluation of $X$ at time $s_1$ via the dynamic risk measure $(\rho_s)_s$ is obtained by regarding $X$ as a $\mathcal{F}_T$-measurable claim and computing $\rho_{s_1}(X)$.\\
On the other side the risk evaluation of $X$ at $s_1$ via the fully-dynamic risk measure $(\rho_{st})_{s,t}$ is given using the risk measure pertaining to the time interval $[s_1,s_2]$, i.e. it is $\rho_{s_1s_2}(X)$. \\
Hence, even if $\rho_{s_1}(X) = \rho_{s_1T}(X)$, we have, in general, that $\rho_{s_1s_2}(X) \ne \rho_{s_1T}(X)$.
Summarising, when using a dynamic risk measure $(\rho_s)_s$, one is implicitly introducing the modelling choice to consider any risk $X$ that is $\mathcal{F}_t$-measurable with $t<T$ as a risk occurring at $T$ and thus evaluating it according to 
$$
\rho_s(X) = \rho_{sT}(X), \qquad s \leq t < T.
$$
It is important to think of the reasonability of this modelling choice particularly when dealing with long time horizons $[0,T]$.
In this paper, this modelling choice is denoted as ``restriction property" for the fully-dynamic risk measure.

From the mathematical analysis point of view, the use of fully dynamic risk measures extends and clarifies the discussion on the concepts of time-consistency with respect to risk evaluations.
To explain briefly, the notion of {\it strong time-consistency} for a fully-dynamic risk measure $(\rho_{st})_{s,t}$ was introduced in \cite{BN01} and there also completely characterised.
In the present work, we introduce the concept of {\it time-consistency} for the fully-dynamic risk measure $(\rho_{st})_{s,t}$ (see Definition \ref{def-fdtc}).
Both concepts are compared with the meaning of time-consistency for the dynamic risk measure $(\rho_{s})_{s}$, up to now mostly frequently used in the literature, see e.g. \cite{AP}. 
In studying this comparison one can appreciate how different is to work with fully-dynamic risk measures compared to dynamic risk measures and the results obtained in the first case cannot be recovered in the second case, without making additional strong a priori assumptions.
We find out that the relationship among all these concepts is delicate and details are studied in this work, see Proposition \ref{Prop:relationtc} and Corollary \ref{Corollary2.5}. 

\vspace{2mm}
Hereafter we give a brief overview of the structure and results of the present work.
The paper is organised in two main parts: Section \ref{SecRI} and Section \ref{SecNGD}, which are divided in subsections.

\vspace{2mm}
Section \ref{SecRI} deals with risk-indifference pricing.
After the formal introduction of the framework of convex price operators and convex price systems, we proceed with the specifications of risk-indifference price operators, which are  defined via risk measures.

\vspace{2mm}
In this paper we work with fully-dynamic risk measures $(\rho_{st})_{s,t}$ that are assumed strong time-consistent and which are {\it not} assumed normalised a priori, nor satisfying the restriction property.

\vspace{2mm}
In Subection 2.1 the relevant concepts about $(\rho_{st})_{s,t}$ are introduced and discussed. As already announced, we study the time-consistency and the strong time-consistency, as well as the relationship between fully-dynamic risk measures and dynamic risk measures.
Furthermore, the two concepts of a {\it dominated} and a {\it sensitive} risk measure are introduced with the corresponding results. The subsection is completed by some examples of strong time-consistent fully-dynamic risk measures.

\vspace{2mm}
Subsection 2.2 deals with risk-indifferent pricing and the price operator $x_{st}$ is introduced by means of $\rho_{st}$ belonging to the strong time-consistent fully-dynamic risk measure $(\rho_{st})_{s,t}$.
The properties of $x_{st}$ are studied to see if the risk-indifference price is indeed a convex price operator.

In our study we find that for $p\in [1,\infty)$ under mild assumptions, a risk-indifference price operator $x_{st}$ is a mapping from $L_p(\F_t)$ into the space of $\F_s$-measurable random variables, but the operators take values in $L_p(\F_s)$ only on a restricted domain $L_t:= Dom \, x_{st} \subsetneq L_p(\F_t)$. 

On the contrary, if $p=\infty$, then the operator $x_{st}$ is well-defined on $L_\infty(\F_t)$ with values in $L_\infty(\F_s)$.
The study of the Fatou property (and, as we see in Subsection 2.3, also the study of time-consistency for the family $(x_{st})_{s,t}$) is then far from trivial when $p<\infty$.
We tackle the issue by extending the operator $x_{st}$ from $Dom \, x_{st} \subsetneq L_p(\F_t)$ to an appropriate Banach space $L^c_t$. 
To achieve such an extension, which maintains the same structure of a risk-indifference operator, we have first to extend the risk measures themselves. 
We stress that we have to perform the extension of the \emph{whole} family of risk measures in a careful way, so to obtain that the new extended fully-dynamic risk measure is strongly time-consistent.
For this we make use of the notions of \emph{sensitive} and \emph{dominated} risk measures.
The domination is inspired by sandwich extension theorems as in \cite{BNDN2}
 and it is used in \cite{BNK} to obtain a representation of risk measures. 
 The sensitivity is linked to the relevance as introduced in \cite{KS} and it allows to write the representation of risk measures in terms of $P$-equivalent probability measures. See \cite{FS04}. 
The extended risk-indifference operator is obtained directly from the extended risk measures. It is a mapping defined on $L^c_t$, it takes values in $L^c_s$, and it fulfils the requirements to be a convex price operators.

Here we also emphasise that the use of strong time-consistency of $(\rho_{st})_{s,t}$ is crucial in obtaining the not-standard representations of the risk measures in Proposition \ref{Prop2rep} (see equations \eqref{Rep2}, \eqref{Rep3}), which are in their turn a crucial element for the arguments proposed later. Also, it is the strong time-consistency that allows to assume sensitivity and dominatoion only on the risk measure $\rho_{0T}$ so to obtain the above mentioned representations and subsequent extensions (see Theorem \ref{thmextrho}).

\vspace{2mm}
Subsection 2.3 deals with the whole family of risk-indifference price operators $(x_{st})_{s,t}$ defined as mappings $x_{st}$ on $L^c_t$ with values in $L^c_s$.
Here we study the time-consistency of $(x_{st})_{s,t}$ and we find that the family of operators $(x_{st})_{s,t}$ on $(L^c_t)_t$ is {\it time-consistent} and thus it is a {\it convex price system}, according to Definition \ref{defCPS}.
By this we can conclude that working with risk indifference pricing from a dynamic perspective is actually a delicate affair and this should be kept in mind when studying techniques of computation of the prices themselves.

We complete the study showing that, for all $t$ and $X\in L^c_t$, the processes $x_{st}(X)$, $s\in [0,t]$, admit c\`adl\`ag modification. This also ensures that stopping times can be considered.

\vspace{2mm}
Working with the Banach spaces $(L^c_t)_t$ is mathematically effective, though not straightforward. For this reason in the last part of the paper, Section \ref{SecNGD}, we propose to work with no-good-deal bounds on prices, so to obtain a risk-indifference price system $(x_{st})_{s,t}$ with domain $(L_2(\F_t))_t$. No-good-deal bounds have been introduced in a dynamic setting in \cite{BNDN}. 
Here we study first the representation of \emph{convex no-good-deal prices} and then we combine this with the risk-indifference prices.

Our final result gives a characterisation of those risk measures for which the risk-indifference prices are actually no-good-deal.


\section{Risk-indifference pricing}
\label{SecRI}

The claims and all financial positions belong to an $L_p$-space for $p\in [1,\infty]$. We study first the properties of $x_{st}$ for each $s\leq t$ to verify under what circumstances $x_{st}$ is a proper convex price operator, and then we study the time-consistency of the whole family $(x_{st})_{s,t}$ to obtain a convex price system.

\vspace{2mm}
Specifically, for the time horizon $[0,T]$, with $T \in (0,\infty)$, let $(\Omega,\F, ({\cal F}_t)_t,P)$ be a complete filtered probability space. 
We assume that $\F_0$ is trivial.
For any $p \in [1,\infty]$ and a $\sigma$-algebra $\B$ of events of $\Omega$, we consider the space $L_p(\B):=L_p(\Omega,\B, P)$ of real valued random variables with the finite norms:
\[
\Vert X \Vert_p := 
\begin{cases}
(E[ \vert X\vert^p ] )^{1/p}, \qquad& p \in [1, \infty),\\
\esssup \vert X \vert, \qquad& p = \infty.
\end{cases}
\]
We denote the expectation under $P$ by $E$. If other measures are used they will be specified.
For any time \(t\in[0,T]\), we consider the \emph{convex} subset:
\begin{equation}
\label{inclusion}
L_t\subseteq L_p(\F_t) : \quad L_t\subseteq L_T.
\end{equation}

The prices are in general defined on the domain $L_t$ in the set of purchasable assets. Note that, in general, \(L_t \subseteq L_p(\F_t)\) for some \(t\in [0,T]\).
In view of the financial application, it is justified to require that:
\begin{enumerate}
\item[A1)]
$1 \in L_t$,
\item[A2)]
for the $\sigma$-algebra $\F_t $ we have the property that $1_A X \in L_t$ for every $A \in \F_t$ and every $X \in L_t$.
\end{enumerate}

To take care of the value of money over time, we consider a num\'eraire of reference $R_t$, $t \in [0,T]$, used as discounting factor.
This refers to a money market account. In this work we set $R_t \equiv 1$ to symplify the notation as we are dealing with discounted values. Hence we shall just write about \emph{prices}, though effectively, we refer to \emph{discounted prices}.

\bigskip
\begin{definition}
\label{convex-price}
For any $s,t \in [0,T]$ with $s \leq t$, the operator 
\begin{equation}
\label{operator}
x_{st}: L_t \longrightarrow L_s
\end{equation}
is a \emph{{\bf (convex) price operator}} if it is:
\begin{itemize}
\item
\emph{monotone}, i.e. for any $X',\,X''\,\in L_t$ such that $X'\geq X''$,
\begin{equation*}
  x_{st}(X')\geq   x_{st}(X''),\label{mono}
\end{equation*}
\item 
\emph{convex}, i.e. for any $X',\,X''\,\in L_t$ and $\lambda\in [0,1]$, 
\begin{equation*}
  x_{st} \big(\lambda X'+ (1-\lambda) X''\big) \leq \lambda x_{st}(X')+ (1-\lambda) x_{st}(X'')
\label{addi}
\end{equation*}
\item 
\emph{it has the Fatou property}, i.e. for any $X\in L_t$ and any dominated sequence $(X_n)_n$ in $ L_t$ (i.e. there exists $Y\in L_p(\F_t)$ such that 
$\vert X_n \vert \leq Y$ for all $n$) which is converging $P$-a.s. to $X$, we have:
\begin{equation*}
\liminf_{n\to \infty} x_{st}(X_n) \geq x_{st}(X)
\label{lsc}
\end{equation*}
\item 
\emph{ weak \(\F_s \)-homogeneous}, i.e. for all \(X\in L_t\)
\begin{equation*}
  x_{st} ( 1_A X)= 1_A   x_{st}(X),  \quad A \in \F_s,
\label{homo}
\end{equation*}
\item
\emph{projection property}
\begin{equation*}
\label{normal}
x_{st} (X)=X, \quad X\in L_p(\F_s)\cap L_t.
\end{equation*}
\end{itemize}
In particular we have $x_{st}(0)=0$, $x_{st}(1)=1$, and $x_{tt}(X) = X$, $X \in L_t$.
\end{definition}

We remark that the ${\cal F}_s$-homogeneity is called ``locality'' in some other references in the literature, see for example \cite{BCP}.

\bigskip
We remark that, if $p\in [1,\infty)$ and the operator $x_{st}$ is monotone and linear (as in \cite{ADR} and \cite{DE08}), the \emph{weak $\F_s$-homogeneity} is equivalent 
to \emph{$\F_s$-homogeneity}, i.e. for all \(X\in L_t\), $x_{st}(\xi X)= \xi   x_{st}(X)$ for all  \(\xi \in L_p(\F_s)\) such that \(\xi X\in L_t\). 
If $p=\infty$ and the operator is linear and semi-continuous, then the same result holds (see \cite{BNDN}).

\bigskip
For the whole family of price operators we introduce two concepts of 
time-consistency. 
We remark that these concepts are here specifically given for price operators and they are only to be compared, but not the same as those of time-consistency for risk measures (cf. Definition \ref{def-fdtc}).

\begin{defn}\label{def-consist}
The family of convex price operators $x_{st}: L_t \longrightarrow L_s$, $s,t\in [0,T]$: $s \leq t$, is 
\begin{itemize}
 \item 
\emph{strong time-consistent}  if, for all $r,s, t\in [0,T]$: $r \leq s \leq t$, we have that
\begin{equation}\label{consist}
x_{rt}(X)=x_{rs}\big(x_{st}(X)\big), \quad  X\in L_t,
\end{equation}
\item 
\emph{time-consistent}  if, for all $t\in [0,T]$, 
$$
x_{st}(X) \geq x_{st}(Y), \quad \textrm{ for } X,Y \in L_t,
$$
implies that
$$
x_{rt}(X) \geq x_{rt}(Y), \quad \textrm{ for all } r \leq s \leq t.
$$
\end{itemize}
\end{defn}
The notion of \emph{strong time-consistency} is naturally extending the very property of linear prices, which comes from the chaining of conditional expectations. This notion is here distinguished from the \emph{time-consistency}, which is a weaker notion in line with the literature of family of convex operators such as convex risk measures.

We observe that \emph{strong time-consistency} of prices implies the \emph{time-consistency}. This is easily derived by means of the projection property.

\bigskip
Finally, a price system is defined as follows:
\begin{defn}
A \emph{\bf (convex) price system} $(x_{st})_{s,t}$ on $(L_t)_t$ is the family of convex price operators \eqref{operator} if it is 
at least time-consistent.
\label{defCPS}
\end{defn}

\vspace{3mm}
We study dynamic risk-indifference prices in the $L_p$ setting, $p \in [1,\infty]$. 
As we shall see, for the fixed times $s,t: s \leq t$, each risk-indifference price evaluation is substantially a convex price operator.
However, we have to be cautious with the definition in order to guarantee the Fatou property and then the time-consistency of the whole price system.

This section is divided in three parts.
In the first we introduce and study properties of a family of risk measures $(\rho_{st})_{s,t}$ indexed by time to achieve a \emph{fully-dynamic structure}. 
Then we study the corresponding risk-indifference price evaluations $(x_{st})_{s,t}$. Here we discuss the very definition of 
$x_{st}$ both in terms of domain and range. We shall see that while in the case $p=\infty$, $x_{st}$ is a well-defined convex price, 
for the case $p<\infty$ we have to introduce a new framework beyond the $L_p$-setting. 
It is only in this new framework that we can show that the risk-indifference evaluations 
$x_{st}$ have all the properties to be 
a convex price operator Definition \ref{convex-price}. In particular we obtain the Fatou property and the dual representation of $x_{st}$. In the last parts of this section we study the time-consistency of the family $(x_{st})_{s,t}$ and the regularity of the trajectories. 

\bigskip
\subsection{Fully-dynamic risk measures, domination, sensitivity.} 
At first we introduce the notion of \emph{fully-dynamic risk measure}. 
This object was already considered in \cite{BN01} under the name of dynamic risk measure. 
Then we review the different concepts of \emph{time-consistency} and we deal with the notions of \emph{domination} and \emph{sensitivity.}

\vspace{2mm}
\begin{definition}
A \emph{fully-dynamic} convex risk measure $(\rho_{st})_{s,t}$ is a family of convex risk measures indexed by two points in time $s,t$ such that $0 \leq s \leq t \leq T$.

For any two points of time $s \leq t$, the mapping $\rho_{st}: L_{p} (\F_t) \longrightarrow L_{p}(\F_s)$ 
satisfies the following properties:
\begin{itemize}
\item
\emph{monotonicity}, i.e. $\rho_{st}(X') \leq \rho_{st}(X'')$, for any $X' \geq X''$ in $L_{p}(\F_t)$,
\item
\emph{convexity}, i.e. for any $X', X'' \in L_{p}(\F_t)$ and $\lambda \in [0,1]$, $\rho_{st}(\lambda X' + (1-\lambda)X'') 
\leq \lambda \rho_{st}(X') + (1-\lambda) \rho_{st}(X'')$,
\item
\emph{$\F_s$-translation invariance}, i.e. for any $X \in L_{p}(\F_t)$, $\rho_{st}(X + f) = \rho_{st}(X) - f$, for all $f \in L_{p}(\F_s)$.
\end{itemize}
Moreover,
\begin{itemize}
\item
If $p = \infty$, then we assume that \emph{continuity from below} holds with convergence point-wise, i.e. for any 
$X \in L_{\infty}(\F_t)$ and any sequence $(X_n)_n$ such that 
$X_n \uparrow X$ $P$-a.s., then $\rho_{st} (X_n) \downarrow \rho_{st}(X)$ $P$-a.s. 
\end{itemize}
\end{definition}

\begin{remark}
If $p <\infty$, we recall that the risk measure $\rho_{st}$ is always \emph{continuous from below and continuous from above} $P$-a.s. (and in the $L_{p}$-convergence).
In fact, recall that for any sequence $X_n$ increasing to $X$, the sequence $\rho_{st}(X_n)$ is decreasing.  Define $\rho(X):=E(\rho_{st}(X))$. Then $\rho$ is a static risk measure. Applying the results of \cite{BF2009}, \cite{FiSv}, we have that $\rho(X_n) \downarrow \rho(X)$. The above mentioned monotonicity of $\rho_{s,t}$ implies that 
$\rho_{st}(X_n) \downarrow \rho_{st}(X)$ $P$-a.s., which in turns implies also the $L_p$-convergence.
\end{remark}

We stress that we \emph{do not assume a priori} that the risk measures $\rho_{st}$ are normalised.\\
We observe that the risk measure $\rho_{st}$  satisfies \emph{weak $\F_s$-homogeneity}, i.e.
\begin{equation}
 1_A\rho_{st}(X)=1_A\rho_{st}(1_AX), \qquad   X \in L_{p}(\F_t),\quad A \in \F_s.
 \label{eqwh}
\end{equation}
Indeed for $p=\infty$, monotonicity and translation invariance  imply (\ref{eqwh}) for all $X \in L_{\infty}(\F_t)$. 
See \cite{CDK2006} Proposition 3.3 and \cite{KS} Section 3.
In the case $p<\infty$,  the weak $\F_s$-homogeneity for all $X \in L_{p}(\F_t)$ is then guaranteed by the continuity from below.

\bigskip
\begin{proposition}
For any $s \leq t$ the risk measure above admits the following representation:
\begin{equation}
\label{rm-representation0}
\begin{split}
\rho_{st} (X) &=  \underset{\substack{Q \ll P : \\  Q_{\vert \F_s} = P_{\vert \F_s}}}{\essmax} \Big(E_Q[-X\vert \F_s] - 
\alpha_{st}(Q)\Big)  \\
&= \underset{\substack{Q \ll P : \\  E_P [\alpha_{st}(Q)] < \infty }}{\essmax} \Big(E_Q[-X\vert \F_s] - \alpha_{st}(Q)\Big) , \quad X \in L_{p}(\F_t),
\end{split}
\end{equation}
where $\alpha_{st}$ is \emph{the minimal penalty}, i.e.
\begin{equation}
\label{minpen}
\alpha_{st}(Q) := \underset{X \in L_{p}(\F_t)}{\esssup} \Big(E_Q [-X\vert \F_s] - \rho_{st}(X) \Big).
\end{equation}
\end{proposition}

\begin{proof}
The arguments are substantially given in \cite{BN-preprint} (written for the case $p = \infty$). Here we propose a self-contained proof.

The inequality 
$\rho_{st}(X) \geq \underset{\substack{Q \ll P : \\  Q_{\vert {\cal F}_s} = P_{\vert {\cal F}_s}}}{\esssup} \Big(E_Q[-X\vert {\cal F}_s]- \alpha_{st}(Q)\Big)   $ 
follows from the definition of the minimal penalty.

Define $\rho(X):=E(\rho_{s,t}(X))$,  $\rho: L_p({\cal F}_t)  \rightarrow  \R$. 
Then $\rho$ is a static risk measure, continuous from below. The representation with a max 
\begin{equation}
\rho (X) = \underset{\substack{Q \ll P }}{\max} (E_Q(-X) - \alpha(Q)) 
\label{eqmax0}
\end{equation}
where $\alpha(Q)$ is the minimal penalty for $\rho$
follows from \cite{FS04} in the case $p=\infty$ and from \cite{FiSv} when $p<\infty$.
As in the proof of the dual representation in the case of continuity from above, (cf \cite{DS2005} or \cite{AP}), $\alpha(Q)<\infty$ implies  that $Q \ll P$, $Q_{\vert \F_s} = P_{\vert \F_s}$ and $E(\alpha_{st}(Q) \leq \alpha(Q)$.
For any given $X$, let $Q_X$ be an argmax in (\ref{eqmax0}).
Then we have
\begin{align*}
\rho(X)=&E_{Q_X}(-X)-\alpha(Q_X)) \\
\leq& E[E_{Q_X}(-X|{\cal F}_s)-\alpha_{st}(Q_X)] \\
\leq& E(\rho_{st}(X))=\rho(X)
\end{align*}
This proves that $E_{Q_X}(-X|{\cal F}_s)-\alpha_{st}(Q_X)=\rho_{st}(X)$.
\end{proof}

\begin{remark}
\label{strong convexity}
We observe that the risk measures $\rho_{st}$ satisfy a stronger convexity property, i.e. for any $X,Y \in L_p(\F_t)$ it is
$$
\rho_{st}(\Lambda X +(1-\Lambda)Y) \leq \Lambda \rho_{st}(X) + (1-\Lambda)\rho_{st}(Y), \quad \Lambda \in L_\infty(\F_s), 0 \leq \Lambda \leq 1.
$$
The property follows from \cite{KS} Remark (2), page 602. 
\end{remark}
 
\begin{lemma}
\label{lemmaA}
Given the representation \eqref{rm-representation0}, let $Q$ be a probability measure such that 
$E(\alpha_{st}(Q))< \infty$. 
Then  we have 
\begin{eqnarray}
\alpha_{st}(Q)
&=& \underset{X \in L_{\infty}({\cal F}_t)}{\esssup } \:\Big(E_Q(-X|{\cal F}_s)- \rho_{st}(X)\Big).
\label{peninfty}
\end{eqnarray}
\end{lemma}

\begin{proof} 
We only need to prove \eqref{peninfty} in the case where $p<\infty$.
It is  well known that 
$\alpha_{st}(Q) =
\esssup_{X \in {\cal A}^{p}_{st}} E_Q(-X|{\cal F}_s)$ where ${\cal A}^{p}_{st}$ is the acceptance set for $\rho_{s,t}$, i.e.,
${\cal A}^{p}_{st}: =\{X \in L_{p}({\cal F}_t)| \rho_{st}(X) \leq 0\}$. See \cite{DS2005}.
Let $X \in {\cal A}^{p}_{st}$. For all $n>0$,  let $X_n=\sup(X,-n)$. Every $X_n$ belongs to ${\cal A}^{p}_{st} 
\subseteq L_{p}(\F_t)$. 
The sequence $-X_n$ is increasing to $-X$, it follows from the monotone convergence theorem that $E_Q(-X_n|{\cal F}_s)$ 
is increasing to  $E_Q(-X|{\cal F}_s)$. 
Then the continuity from above of $\rho_{st}$ allows to restrict the evaluation of $\esssup_{X \in L_{p}({\cal F}_t)}[E_Q(-X|{\cal F}_s)- \rho_{st}(X)]$ to the elements $X$ bounded from below. 
So, let $X \in  L_{p}({\cal F}_t)$ be bounded from below by $C$. 
Then $X $ is the increasing limit of the sequence of bounded random variables $Y_n= \inf(X,n)$. 
Moreover, $|Y_n| \leq \sup (|C|,X)$.
 The hypothesis $E(\alpha_{st}(Q))< \infty$  implies that $Q \ll P$ with a density in $L_{ q}({\cal F}_t)$, where 
 $ q$ is the conjugate of ${p}$. It follows from the dominated convergence theorem for conditional 
 expectations that 
 $E_Q(-X|{\cal F}_s)=\lim_{n \to \infty} E_Q(-Y_n|{\cal F}_s)$. 
 On the other hand, $\rho_{st}$ is continuous from below, thus $\rho_{st}(X)=\lim_{n \to \infty} \rho_{st}(Y_n)$.
 This proves that 
$$
E_Q(-X|{\cal F}_s)- \rho_{st}(X) \leq \underset{X \in L_{\infty}({\cal F}_t)}{\esssup} [E_Q(-X|{\cal F}_s)- \rho_{st}(X)].
$$ 
The result follows easily. 
\end{proof}

As announced, fully-dynamic risk measures $(\rho_{st})_{s,t}$ allow the evaluation of the risks related to the time-interval $[s,t]$ to be measured by $\rho_{st}$, which is related to the very time-interval. Indeed a $\F_t$-measurable position $X$, with $t<T$, has risk evaluation at $s<t$ given by $\rho_{st}(X)$.
We recall that when using a dynamic risk measure $(\rho_s)_s$, the evaluation of $X$ would be done regarding $X$ as a $\F_T$-measurable position (this regardless how far away in the future is $T$ compared to $t$) and thus obtaining the evaluation $\rho_s(X)$.
Indeed, even if $\rho_s(X) = \rho_{sT}(X)$, it may be $\rho_{st}(X) \ne \rho_{sT}(X)$.

So, fully-dynamic risk measures allow for the possibility of modelling changes in the rules of the evaluation along with time. 
This is particularly reasonable whenever the time horizon $T$ is large and the state of the economy are more likely to change.

\begin{remark}
The idea of having an evaluation criteria evolving with time is also identified in \cite{MZ2}, where the concept of {\it dynamic performance} is detailed together with the one of {\it forward performance}, i.e. a dynamic performance with a specification at the initial time. These are discussed in the context of portfolio optimisation and indifference prices in markets with explicitly modelled price dynamics, e.g. the binomial model in \cite{MZ1} and an It\^o type dynamic in \cite{MZ2}. The dynamic performance is intended to extend the concept of {\it value process} generated (backwards) by a utility function typically by dynamic programming. The dynamic performance $(U_t(x))_t$ is in fact only capturing those essential properties that such value process posses: the adaptedness to a given information flow, the supermartingality of
$(U_t(X))_t$ for any attainable $X$, and the martingality of $(U_t(X^*))_t$, for an optimal and attainable $X^*$.
We remark that the dynamic performance $(U_t(x))_t$ is characterised by {\it one} time index only. Also, in general, for $s<t$ and an $\F_s$-measurable $X$, $U_s(X) \ne U_t(X)$. In this they differ from dynamic risk-measures as fully-dynamic risk measures do (i.e. there is no {\it \`a priori} restriction property \eqref{restriction}).
On the other side, a dynamic convex risk measures plus its minimal penalty $V_t(X) := \rho_t(X) + \alpha_t(Q)$, $t\in [0,T]$, is a $Q$-supermartingale, for $X$ fixed, and it is a $Q$-martingale if the maximum is achieved at $Q$ in the robust representation of $\rho_0(X)$. See \cite{AP}.
\end{remark}

\subsubsection{About time-consistency.}

Hereafter we give two concepts of time-consistency for fully-dynamic risk measures. 

\begin{defn}
\label{def-fdtc}
A fully-dynamic risk measure on $L_p$ is 
\begin{itemize}
 \item 
\emph{ strong time-consistent} if for all $r,s,t \in [0,T]: r \leq s \leq t$, we have
$$ 
\rho_{rt}(X) = \rho_{rs} (-\rho_{st} (X)), \quad   X \in L_{p}(\F_t),
$$
\item \emph{ time-consistent} if for all $r,s,t \in [0,T]: r \leq s \leq t$,  for all $X,Y \in L_{p}(\F_t)$, we have
$$ 
\rho_{st}(X) = \rho_{st} (Y) \Longrightarrow \rho_{rt}(X) = \rho_{rt} (Y).
$$
\end{itemize}
\end{defn}
The second concept, here called simply \emph{time-consistency}, is deriving directly from the same notion introduced for classical dynamic risk-measures $(\rho_s)_s$, indexed by one time parameter, see \cite{AP} and see e.g. \cite{bielecki} for a recent survey on various weaker concepts of time-consistency. 
\begin{remark}
\label{tc as order}
Time consistency is substantially an ``order preserving property''. Indeed, from monotonicity and translation invariance, we have that time consistency as defined above is equivalent to the order preserving property: {\it For all $r \leq s \leq t$, $\rho_{st}(X) \geq \rho_{st}(Y)$, implies $\rho_{rt}(X) \geq \rho_{rt}(Y)$.}
\end{remark}

\begin{remark}
\label{tc and normalisation}
It is easy to see that time-consistency is transferred to the normalised version of $(\rho_{st})_{s,t}$ and, more generally, to any translation of $(\rho_{st})_{s,t}$.
\end{remark}

The first concept, here called \emph{strong time-consistency} to distinguish it from the other one, was already introduced in \cite{BN01}, where a complete characterisation was also provided.
\begin{remark}
\label{charcoc}
Strong time-consistency is a ``composition rule''. It means that for all $r \leq s \leq t$, the risk at time $r$ associated to  any random variable $X$ ${\cal F}_t$-measurable can be computed either directly as $\rho_{r,t}(X)$ or making use of the instant of time $s$ as $\rho_{r,s}(-\rho_{s,t}(X))$, and both quantities are the same. 
We recall that the strong time-consistency is fully characterised in terms of cocycle condition of the the minimal penalties. See Theorem 1 in \cite{BN02}.
\end{remark}

\begin{remark}
\label{stc and normalisation}
We observe that the property of strong time-consistency is {\it not} transferred to the normalised version $(\check\rho_{st})_{s,t}$ of $(\rho_{st})_{s,t}$. \\
To see this, assume that the fully-dynamic risk measure $(\rho_{st})_{s \leq t}$ satisfies strong time-consistency. We know that this is equivalent to the cocycle condition. 
From the  definition of the minimal penalty \eqref{minpen} we have that the associated normalised fully-dynamic risk measure ($\check \rho_{s,t})$ satisfies also strong time-consistency if and only if, for all probability measure $Q$ such that 
$E_Q(\alpha_{rt}(Q))<\infty$, 
$$\rho_{rt}(0)=\rho_{rs}(0)+E_Q(\rho_{st}(0)|{\cal F}_r).$$
This condition is, in fact, very strong.\\
\end{remark}

The following result details the relationship between the two forms of time-consistency. One can compare this with the concepts of time-consistency for price operators as in Definition \ref{def-consist}.
\begin{proposition}
\label{Prop:relationtc}
Let $\rho_{st}:  L_p({\cal F}_t) \longrightarrow L_p({\cal F}_s)$,  $s,t \in [0,T]:s \leq t$. 
The following assertions are equivalent:
\begin{itemize}
 \item [i)] 
The fully-dynamic risk measure $(\rho_{st})_{s,t}$ is \emph{strong time-consistent}.
 \item [ii)] The fully-dynamic risk measure  $(\rho_{st})_{ s,t}$ is  \emph{time-consistent}  
 and 
 \begin{equation}
 \rho_{rt}(Y)=\rho_{rs}(Y-\rho_{st}(0)),\quad 0 \leq r \leq s \leq t, \;\;  Y \in L_p({\cal F}_s).
 \label{eqwstc}
 \end{equation}
      \end{itemize}
     \end{proposition}

 \begin{proof}
To see that (i) implies (ii), it is enough to use the translation invariance property of $\rho_{st}$.

Conversely, assume that (ii) is satisfied. Let $Z \in L_p({\cal F}_t)$ and define $Y:=\rho_{st}(0)-\rho_{st}(Z)$. 
From the translation invariance property of $\rho_{st}$ follows that 
$\rho_{st}(Y)=\rho_{st}(0) - [\rho_{st}(0) - \rho_{st}(Z)]= \rho_{st}(Z)$. Thus, from time-consistency, $\rho_{rt}(Y)=\rho_{rt}(Z)$. Applying (\ref{eqwstc}) to $Y$, we get that $ \rho_{rt}(Z) 
= \rho_{rs} (-\rho_{st}(Z))$. 
\end{proof}
     
In the case when the risk measures $(\rho_{st})_{s,t}$ are normalised, the above result has an easy interpretation.
\begin{corollary}
\label{Corollary2.5}
 Assume that the fully-dynamic risk measure $(\rho_{st})_{s,t}$   is normalised i.e. $\rho_{st}(0)=0$, for all $s \leq t$. 
 Then the  fully-dynamic risk measure $(\rho_{st})_{s,t}$  is strong time-consistent if and only if it is time-consistent and satisfies the \emph{restriction property}, i.e. for all $0 \leq r \leq s \leq t$,
 \begin{equation}
 \label{restriction}
 \rho_{rt}(Y)=\rho_{rs}(Y), \qquad Y \in L_p({\cal F}_s).
 \end{equation}
 \end{corollary}

\bigskip
To summarise, in this work we consider a  strong time-consistent fully-dynamic risk measure $(\rho_{st})_{s,t}$ not necessarily normalised. Also we do not assume the restriction property \eqref{restriction}.

\subsubsection{About domination.}

Hereafter we consider the concept of domination for a risk measure, first introduced in \cite{BNK}. This is going to be crucial for defining the correct framework to study risk-indifference 
price operators.
\begin{defn}
\label{def-domination} 
Fix $p < \infty$. Let $\rho: L_p(\F_T) \longrightarrow \mathbb{R}$ be a convex risk measure continuous from below.
The risk measure $\rho$ is \emph{dominated} if there exists a sublinear (or coherent) risk measure $\tilde \rho: L_{p}(\F_T) \rightarrow \R$ 
such that 
$$
\rho(X)-\rho(0) \leq \tilde \rho(X), \qquad  X \in L_{p}(\F_T).
$$
\end{defn}

In \cite{BNK} the property of domination is characterised and it is proved that it guarantees a representation result.
The following proposition summarises these findings.

\begin{proposition}
\label{propRepBNK}
Fix $p\in [1,\infty)$. Let $\rho: L_{p}(\F_T) \longrightarrow \R$ be a convex risk measure. The following statements are equivalent:
\begin{enumerate}
 \item 
The risk measure $\rho$ is dominated.
\item 
There exist $K>0$ and $C\in \mathbb{R}$ such that $\rho(X) \leq K \Vert X \Vert_p +C$, for all $X \in L_p(\F_T)$.
\item 
The risk measure $\rho$ admits representation
\begin{equation}
 \rho(X)=\max_{Q \in {\cal B}^{K}} (E_Q(-X)-\alpha(Q)), 
 \qquad X \in L_p(\F_T),
 \label{eqnrepBNK}
\end{equation}
where ${\cal B}^K$ is the set of probability measures on ${\cal F}_T$:
$$
{\cal B}^K:=\Big\{Q \ll P:  \frac{dQ}{dP} \in L_q({\cal F}_T), \;||\frac{dQ}{dP}||_q \leq K\Big\},\;\;q=p(p-1)^{-1},
$$
and $\alpha$ is the minimal penalty for $\rho$.
\end{enumerate}
\end{proposition}

\begin{proof}
When $\Omega$ is a Polish space, the results of \cite{BNK} apply directly to 
$L_p(\Omega,{\cal B}(\Omega), P)$ (with $c(f)=E(|f|^p)^{1/p}$).\\
However, Proposition 3.1 of \cite{BNK} is also valid for $L_p(\Omega,{\cal F},P)$ without any assumption on 
$(\Omega,{\cal F},P)$. Indeed the representation 
$$\rho(X)=\underset{Q \in L_p({\cal F})^*}\max \big(E_Q(-X)-\alpha(Q)\big)$$
follows from the extended Nomiaka Klee Theorem (\cite{BF2009} Theorem 1). Replacing everywhere in the proof of \cite{BNK} Proposition 3.1 $L^1(c)$ by $L_p(\Omega)$ and $c(f)$ by $||f||_p$, we obtain the result for general $L_p({\cal F}_T)$.
\end{proof}

The above representation is known in the case of sublinear (or coherent) risk measures, see \cite{FiSv} and \cite{KR}.
Also a representation result for finite convex risk measures on $L_p$ ($p<\infty$) is  given in \cite{KR} Theorem 2.11. This proof is unfortunately based on a wrong statement (see Proposition 2.10 in \cite{KR}).

\bigskip
Note that, if a convex risk measure $\rho$ is dominated, then we have the sandwich:
$$
-\tilde \rho(-X) \leq \rho(X)-\rho (0) \leq \tilde \rho(X), \qquad X\in L_p(\F_T),
$$ 
where $-\tilde \rho (-X)$ is superlinear and $\tilde \rho$ is sublinear.

The sandwich relationship above provides a motivation itself for the use of the property of domination. This is in view of the link with the extension theorem of strong time-consistent convex operators satisfying a sandwich condition as studied in \cite{BNDN2}. 
Indeed if we consider a strong time-consistent family of operators $(\rho_{st})_{s,t}$ 
defined on the vector subspaces $(L_t)_t$ with  $L_t \subseteq L_{p}(\F_t)$ and 
if each $\rho_{st}$ satisfies a sandwich condition and the Fatou property, 
then we can extend this family $(\rho_{st})_{s,t}$ to the whole $(L_{p}(\F_t))_t$. 

\bigskip
\subsubsection{About sensitivity.}

Now we consider the sensitivity of a risk measure. This concept 
yields a representation of the risk measure in terms of probability measures equivalent to $P$. See e.g. Section 3 in \cite{KS}. 
\begin{defn}
Let $\rho: L_p(\F_T) \longrightarrow \mathbb{R}$ be a convex risk measure.
\begin{itemize}
\item
The risk measure $\rho$ is \emph{strong sensitive or relevant (to $P$)}, if
$$ 
\rho(1_B) < \rho(0), 
$$
for all $B \in \F_T$ such that $P(B)>0$.
\item
The risk measure $\rho$ is \emph{ sensitive (to $P$)} if there exists a probability measure $\tilde Q \sim P$ such that $\alpha(\tilde Q) < \infty$, where $\alpha$ is the minimal penalty associated to $\rho$. 
\end{itemize}
\end{defn}

\begin{remark}
The property of relevance implies  sensitivity. This follows from Lemma 3.4 in \cite{KS} applied to $\Phi(X)= - \rho(X)$, restricted to $L_{\infty}(\F_T)$.
\end{remark}

\begin{proposition}
\label{prop1}
Fix $p<\infty$. Let $\rho: L_p(\F_T) \longrightarrow \mathbb{R}$ be a sensitive convex risk measure, then the risk measure $\rho$ admits representation
\begin{equation}
\rho(X) = \sup_{Q \in \mathcal{Q}} \Big( E_Q(X) - \alpha(Q) \Big), \quad X \in L_p(\F_T),
\label{rep0T}
\end{equation}
where
$$
\mathcal{Q} := \Big\{ Q \sim P\,: \quad \alpha(Q) < \infty \Big\}.
$$
Moreover, if $\rho$ is dominated, then 
$$
\mathcal{Q} \subseteq {\cal B}^K
$$
where $K$ is the constant in the domination property and $q=p(p-1)^{-1}$.
\end{proposition}

\begin{proof}
The representation is a direct consequence of Theorem 3.1 in \cite{KS}.
Let $Q \in \mathcal{Q}$, then the  domination implies that $E_Q(X) \leq \alpha(Q) + K \Vert X \Vert_p + C$.
Applying this to $\lambda X$ for all $\lambda >0$ we have
$$
E_Q[X] \leq \frac{\alpha(Q) + C}{\lambda} + K \Vert X \Vert_p.
$$
By taking $\lambda \to \infty$ we conclude.
\end{proof}

\subsubsection{Example: Strong time-consistent fully-dynamic risk measures from BMO martingales}

  Examples of strong time-consistent fully-dynamic risk measures on $L_{\infty}$ were constructed in \cite{BN01} making use of  BMO martingales, cf. Proposition 4.13 and Proposition 4.19. Here we present the $L_p$, $p<\infty$, case. 
  Recall the following definition of BMO martingales.
  \begin{definition}
   \label{BMO}
   A right continuous uniformly integrable martingale $M$ is BMO if there is a constant $C$ such that for any stopping time $S$, 
   $$E([M,M]_{\infty}-[M,M]_{S^-}|{\cal F}_S) \leq C^2.$$
   The smallest $C$ is, by definition, the BMO norm of $M$: $||M||_{BMO}$.
  \end{definition}
For the study of right-continuous BMO martingales, we refer to \cite{DDM} and for continuous BMO martingales we refer to \cite{Ka}. In the case of continuous martingales, the above norm $||M||_{BMO}$ is 
often written $||M||_{BMO_2(P)}$.\\
We prove now that the construction detailed in \cite{BN01} gives also rise to strong time-consistent fully-dynamic risk measures on $L_p$ spaces when the set of martingales used for the construction of 
the dynamic risk measure is a  set of BMO martingales uniformly bounded. \\
Note that for given $K>0$, the set 
$${\cal M}_K=\{M:\; BMO \; \text{martingale}, \; ||M||_{BMO} \leq K\}$$ is a stable set.\\ 
We first prove the following lemma for conditional expectations with respect to the probability measures $Q_M \ll P$ with Radon Nykodym derivative
$$
\frac{dQ_M}{dP} = {\cal E}(M)_T \quad\textrm{ on } {\cal F}_T
$$
associated to some continuous BMO martingale $M$.
 
\begin{lemma}
\label{lemmaBMO}
 Let $K>0$. There is $p(K)>0$, depending only on $K$, such that for all continuous BMO martingale $M$ with 
 $||M||_{BMO} \leq K$, 
the conditional expectation $E_{Q_M}(.|{\cal F}_s)$ defines a continuous operator:
$$
E_{Q_M}(.|{\cal F}_s): L_p(\F_T) \longrightarrow L_p(\F_s),
\qquad p>p(K).
$$
Furthermore the norm of this operator is bounded by some constant
 depending only on $K$ and $p$.  
\end{lemma}
\begin{proof}
 From Theorem 3.1 of \cite{Ka}, there exists  $q(K) \in (1,\infty)$, such that for $q< q(K)$, for all $M$ with $||M||_{BMO} \leq K$, the stochastic exponential ${\cal E}(M)$ satisfies the reverse H\"older inequality:
 \begin{equation}
  \label{eqRH}
  E({\cal E}(M) ^q_t|{\cal F}_s) \leq C_q {\cal E}(M) ^q_s
 \end{equation}
for some constant $C_q$ depending only on $K$ and $q$. \\
Let $p>p(K)= \frac{q(K)}{q(K)-1}$.  It follows from H\"older inequality for conditional expectations  and (\ref{eqRH}) that, for all $X \in L_p({\cal F}_T)$,
\begin{align*}
 |E_{Q_M}(X|{\cal F}_s)| 
 \leq & 
 \big[\big( E(\frac{{\cal E}(M)_T}{{\cal E}(M)_s}\big)^q|
 {\cal F}_s \big)\big]^{\frac{1}{q}} E(|X|^p|{\cal F}_s)^{\frac{1}{p}} \nonumber\\
  \leq & C_q^{\frac{1}{q}} E(|X|^p|{\cal F}_s)^{\frac{1}{p}}.\nonumber
\end{align*}
From this we have 
$$E(|E_{Q_M}(X|{\cal F}_s)|^p)^{\frac{1}{p}} \leq C_q^{\frac{1}{q}}||X||_p.$$
\end{proof}

This allows for the following construction of fully-dynamic risk measures on $L_p$.

\begin{proposition}\label{BMOcont}
 Let ${\cal M}$ be a stable set of BMO continuous martingales with BMO norm bounded by $K$. Let $(b_u)$ be a bounded predictable process. For all $0 \leq s \leq t \leq T$, let
 \begin{equation}
  \label{eqdefFD}
  \rho_{st}(X)=
  \underset{M \in {\cal M}}{\esssup}\; \big(E_{Q_M}(-X|{\cal F}_s)-\alpha_{st}(Q_M)\big)
 \end{equation}
where $Q_M$ is as above, and 
$$\alpha_{st}(Q_M)=E_{Q_M}\Big(\int_s^t b_u d[M,M]_u|{\cal F}_s\Big).$$
Then $(\rho_{st})_{s,t}$ defines a fully-dynamic risk measure on $(L_p(\F_t))_t$ for every $p(K)<p< \infty$, with $p(K) \in (1, \infty)$ depending only on $K$. This fully-dynamic risk measure 
is strong time-consistent. 
Furthermore $\rho_{0T}$ is dominated and sensitive.
 \end{proposition}
 
\begin{proof}
It is proved in the proof of Proposition 4.13 of \cite{BN01} that for all $M$, 
$$||\alpha_{st}(Q_M)||_{\infty} \leq \Big(\sup_u ||b_u||_{\infty}\Big)\, ||M||_{BMO_2(Q_M)}.$$ 
From Lemma 4.12 of \cite{BN02} and the assumption on  ${\cal M}$, we have  that 
\begin{equation}
 \label{eqse}
\sup_{M \in {\cal M}}||\alpha_{st}(Q_M)||_{\infty} <\infty.
\end{equation}
From Lemma \ref{lemmaBMO}, we obtain that, for all $X \in L_p({\cal F}_t)$,  $$||\rho_{st}( X)||_p \leq C_q^{\frac{1}{q}} ||X||_p+C,$$ for some $C>0$.
This proves that $\rho_{st}$ is well defined on $L_p(\F_t)$ with values on $L_p(\F_s)$. Furthermore, the above equation proves that $\rho_{0T}$ is dominated.  The sensitivity of $\rho_{0T}$ follows 
from (\ref{eqse}) and the observation that every $Q_M$ is equivalent to  $P$ on ${\cal F}_T$.\\
It is also proved in \cite{BN01} that the penalties $\alpha_{st}(Q_M)$ satisfy the local condition and the cocycle condition. Strong 
time-consistency follows then from \cite{BN02}.
\end{proof}
\begin{remark}\hspace{3cm}
 \begin{enumerate}
  \item Making use of arguments presented in \cite{BN01} Section 4.4, the above proof can be adapted to the case of BMO martingales with jumps, considering right-continuous BMO martingales with BMO norm 
  uniformly bounded by $K<\frac{1}{16}$.
  \item Other examples of strong time-consistent fully-dynamic risk measures defined on $L_p$ $1 \leq p \leq \infty$  are constructed in Theorem 3 of \cite{BN03} making use of probability measures solutions to 
  path-dependent martingale problems. These dynamic risk measures are linked to solutions to Path Dependent PDEs.
 \end{enumerate}
\end{remark}

\subsubsection{Example: Strong time-consistent fully dynamic risk measures from BSDEs}
 In this example we consider a $d$-dimensional Brownian motion $W$ on $(\Omega,{\cal F},P)$ and the corresponding
 $P$-completed filtration $({\cal F}_t)_{t}$ generated by $W$.  
 We consider the following BSDE 
 \begin{equation}
\label{eqBSDE}
Y_s=-\xi + \int_s^t g(u,\omega, Z_u) du - \int_s^t Z_u dW_u
\end{equation}
where $0 \leq s \leq t \leq T$, $\xi$ is a ${\cal F}_t$-measurable terminal condition and $g$ is a $\F_t$-adapted driver.
  
\begin{proposition}
 \label{prop1BSDE}
 Assume that $g$ is $\R$-valued, convex in $z$, and uniformly Lipschitz with respect to $z$, i.e. there exists $C>0$, 
 $$|g(u,\omega,z)-g(u,\omega,z')| \leq C||z-z'||$$
 and $E(\int_0^T|g(u,\omega,0)|^2du<\infty$.
 Let $Y_s$ be given by the unique solution to the BSDE (\ref{eqBSDE}).
 Then $\rho_{st}(\xi):=Y_s$ defines a fully-dynamic risk measure on $(L_{2}({\cal F}_t))_t$. 
 This fully-dynamic risk measure is strong time-consistent.
\end{proposition}

\begin{proof}
 Pardoux and Peng proved in \cite{PP} the existence and uniqueness of a solution $(Y_s, Z_s)_{0 \leq s \leq t}$ to (\ref{eqBSDE}) such that 
 $E(\sup_{0 \leq s \leq t } |Y_s|^2) <\infty$ and 
 $E(\int_0^T||Z_u||^2du)<\infty$ for all $\xi \in L_2({\cal F}_t)$.
Proposition 6.7 of \cite{BEK}  yields that $(\rho_{st})_{s,t}$ defines a strong time-consistent fully-dynamic risk measure on 
 $(L_2({\cal F}_t))_{t}$. 
 \end{proof}
 
 \begin{proposition}
 \label{prop2BSDE}
 Assume that $g$ is  $\R$-valued, convex, continuous in $z$ with quadratic growth, i.e. there exists $k>0$ such that $(t,\omega)$ a.e. $|g(u,\omega,z)| \leq k(1+||z||^2)$. 
 Furthermore, assume that $g$ is differentiable in $z$ and that there exists $\;c>0$ such that
 $\frac{\partial g}{\partial z} \leq c (1+||z||)$.
 Let $Y_s$ be given by the unique maximal solution to the BSDE (\ref{eqBSDE}), 
 Then $\rho_{st}(\xi):=Y_s$ defines a fully-dynamic risk measure on $(L_{\infty}({\cal F}_t))_{t}$. 
 This fully-dynamic risk measure is strong time-consistent.
\end{proposition}

\begin{proof}
Kobylanski proved in \cite{Ko} the existence and uniqueness of a maximal bounded solution  to (\ref{eqBSDE})  for all $\xi \in L_{\infty}({\cal F}_t)$.
From Proposition 6.7 of \cite{BEK}, we have  that $(\rho_{st})_{s,t}$ defines a strong time-consistent fully-dynamic risk measure on  $(L_{\infty}({\cal F}_t))_{t}$. 
 \end{proof}
 
Observe that the risk measures constructed here above are \emph{not} normalised, in general. In fact we have the following result.
 
\begin{corollary} 
If in Proposition \ref{prop1BSDE} and in Proposition \ref{prop2BSDE}, we assume in addition that $g(t,\omega,0)=0, \;\; (t,\omega) \;a.e.$, then the corresponding strong time-consistent fully-dynamic risk measures $(\rho_{st})_{s,t}$
are normalised and thus satisfy the restriction property \eqref{restriction}. Thus $(\rho_{st})_{s,t})$ defines a conditional $g$-expectation as introduced by Peng \cite{Peng}.\\
\end{corollary}

\begin{remark}
  Under the assumptions of Proposition \ref{prop2BSDE} (quadratic growth in $z$), $\rho_{st}$ admits a dual representation of the kind (2.14) (cf \cite{BEK}) with a set of
  BMO continuous martingales. However, in general the BMO norms of  these BMO martingales are not uniformly bounded.  Therefore in general it is not possible to find 
  $p<\infty$ 
  such that $(\rho_{st})_{s,t}$ defines a fully-dynamic risk measure on $(L_p({\cal F}_t))_t$.
  \end{remark}

\begin{remark}
We also recall the following result from \cite{DPRG}.
On the given $P$-completed Brownian filtration, every normalised strong time-consistent fully-dynamic risk measure $(\rho_{st})_{s,t}$ is the increasing limit of a sequence of normalised strong time-consistent fully-dynamic risk measures $(\rho^n_{st})_{s,t}$, where each $\rho^n_{st}$ is associated to a BSDE with a convex, Lipschitz driver $g_n$.
  \end{remark}

\bigskip
\subsection{A risk-indifference price operator $x_{st}$.}

Risk-indifferent pricing was introduced in a static set-up  as an alternative pricing technique 
to utility-indifference pricing in incomplete markets.  Instead of considering the agents' attitude to a financial investment in terms of utility functions, the risk-indifference approach uses risk measures. The connection between these two approaches is given by 
the fact that utility-indifference with exponential utility function corresponds substantially to a risk-indifference pricing with entropic risk measure.

\bigskip
Hereafter we consider the whole family of risk-indifference prices $(x_{st})_{s,t}$ generated by a strong time-consistent fully-dynamic risk measure 
$(\rho_{st})_{s,t}$ on
$(L_p(\F_t))_t$, $1 \leq p \leq \infty$. 
Risk-indifference pricing is also studied in \cite{KS} for the case of dynamic risk measures 
 $(\rho_t)_t$ as in Subsection \ref{subs2.2.1} and 
 in the case
 $p=\infty$ only.
 
Our first goal is to identify the conditions under which a risk-indifferent price evaluation $x_{st}$, associated to $\rho_{st}$, satisfies the properties for being a convex price operator.

\bigskip
First of all fix $s,t\in [0,T]$: $s\leq t$.
From a qualitative perspective, the risk-indifference (seller's discounted) price $x_{st}(X)$ at time $s$ for any (discounted) financial position $X $, at time $t$,
is given by the equation:
\begin{equation}\label{ri}
\underset{\theta \in \Theta_{st}}{\essinf} \,\rho_{st} (x_{st}(X) +  y_s + Y_{st} (\theta) - X) =
\underset{\theta \in \Theta_{st}}{\essinf}\, \rho_{st} (y_s + 
Y_{st}(\theta)) \quad P-a.s.,
\end{equation}
where 
$y_s\in L_p(\F_s)$ represents the money market account and, together with the price $x_{st}(X)$, is the 
initial capital at $s$. The $\F_t$-measurable 
$Y_{st}(\theta)$ represents the value  of an admissible portfolio $\theta$ on the time horizon $(s,t]$ and the set $
\Theta_{st}$ represents the  admissible portfolios. 
Under suitable integrability conditions, the equation above can be rewritten as
\begin{equation}
\label{eqRIP0}
x_{st}(X)=\underset{\theta \in \Theta_{st}}{\essinf}\, \rho_{st}(Y_{st}(\theta) -X) - 
\underset{\theta \in \Theta_{st}}{\essinf}\, \rho_{st}(Y_{st}(\theta))
\end{equation}
by the translation invariance of $\rho_{st}$. 
Here below we are more specific about  portfolios and value processes, so to achieve a general definition of risk-indifference price.

\vspace{2mm}
Equation \eqref{eqRIP0} makes sense for all those strategies for which $Y_{st}(\theta) \in L_p(\F_t)$, for all $t$. 
In this paper we aim at a definition of risk-indifference price applicable with large generality. Then we proceed with the following set of definitions.

\vspace{2mm}
The market is characterised by a number of underlying assets, whose (discounted) price is given by $(\Pi_{t})_{t \in [0,T]}$ which is a $\Ff$-adapted locally bounded semimartingale in 
$\R^d$.

\begin{defn}
\label{all-strategies}
Denote $\Xi$ the set of all $\Ff$-predictable processes $(\theta_t)_{t\in [0,T]}$ with values in $\R^d$, integrable with respect to the semimartingale $(\Pi_t)_{t \in [0,T]}$ above, and such that for all $0 \leq s \leq T$, the integral process 
$(Y_{st}(\theta))_{t \in [0,T]}$ is $\Ff$-adapted and bounded from below. 
\end{defn}

\begin{defn}
\label{strategies}
The set of \emph{admissible strategies} on $[0,T]$ is constituted by a convex subset  $\Theta := \Theta_{0T} \subseteq \Xi$ satisfying the 
\emph{stability property}:
for any $A \in \F_s$ and any $\theta^{(1)}, \theta^{(2)}, \theta^{(3)} \in \Theta$ the strategy $\theta=(\theta_t)_t$ given by 
\[
\theta_t :=
\begin{cases}
\theta^{(1)}_t,& \quad t \in (0, s]\\
1_A \theta^{(2)}_t + 1_{A^c} \theta^{(3)}_t, &\quad t \in (s, T]
\end{cases}
\quad 
\]
belongs to $\Theta$. Moreover we consider 
$0 \in \Theta$, $Y_{st}(0)=0$. \\
For any $s \leq t$, the set $\Theta_{st}$ of admissible strategies on $(s,t]$ is constituted by all strategies $\theta 1_{(s,t]}$ with $\theta \in \Theta$.
\end{defn}

Clearly if $\Theta = \Xi$, then the stability property is naturally satisfied. 
Our choice to assume $\Theta \subseteq \Xi$ allows for a framework where it is possible to consider exogenous constrains on the applicable strategies. 

\begin{defn}
\label{strategiesC}
For $p \in [1,\infty]$, the sets of \emph{feasible claims} $(\mathcal{C}_{st}^{p})_{s,t}$ are defined by
\[
\mathcal{C}_{st}^{p} := \big\{ g \in L_{p} (\F_t) \, : \quad \exists\, \theta \in \Theta \textrm{ such that } g \leq Y_{st}(\theta) \big\}.
\]
\end{defn}

Note that $0 \in \mathcal{C}_{st}^p$.

\vspace{2mm}
For the risk-indifference price to be well defined, we introduce the following technical assumption,  standing for this paper.

\vspace{2mm}
\noindent
{\bf Assumption:}
\begin{equation}
\label{techass}
\underset{g \in {\cal C}_{st}^{p}} \essinf \rho_{st}(g) > - \infty \qquad P-a.s. 
\end{equation}

\vspace{2mm}
Motivated by the above considerations we give the following definition.
\begin{defn}
\label{defRIP}
Let $p \in [1,\infty]$. For any $s,t\in [0,T]:$ $s\leq t$, the operator 	
\begin{equation}
x_{st}(X): =\underset{g \in {\cal C}_{st}^{p}}{\essinf} \: 
\rho_{st}(g-X) - \underset{g \in {\cal C}_{st}^{p}}{\essinf}\: \rho_{st}(g)
\label{eqRIP}
\end{equation}
is well-defined $P$-a.s. for all $X \in L_p(\F_t)$ as an ${\cal F}_s$-measurable random variable.
We call $x_{st}(X)$ the \emph{risk-indifference price of $X$ given by $\rho_{st}$} when
the operator $x_{st}$ is considered on 
\[
Dom\, x_{st} := \big\{ X \in L_{p}(\F_t)\, : \: x_{st}(X) \in L_{p}(\F_s) \big\}.
\]
\end{defn}
\begin{remark}
Due to the stability property of $\Theta$,  the set $\{\rho_{st}(g-X),\; g \in {\cal C}^{p}_{st}\}$ has the lattice property for all 
$X \in L_{p}(\F_t)$.
\end{remark}
\begin{remark}
We remark that, if the set of strategies $\Theta$ is such that $Y_{st}(\theta) \in L_p(\F_t)$, for all $t$, then $Y_{st}(\theta)\in {\cal C}^{p}_{st}$. By monotonicity of the risk-measures, we can see that \eqref{eqRIP} coincides with \eqref{eqRIP0}.
\end{remark}

\vspace{2mm}
\begin{lemma}\label{Linf}
Let $p \in [1,\infty]$.
\begin{enumerate}
\item
If $\underset{g \in {\cal C}_{st}^{p}}\essinf\rho_{st}(g)$ belongs to  $L_{p}(\F_t)$, then $Dom \, x_{st}=L_{p}(\F_t)$.
\label{remdom}
\item
For all $s\leq t$, we have that, for all $X \in L_\infty(\F_t)$, $x_{st}(X)$ belongs to $L_\infty(\F_s)$. 
In particular, $L_\infty(\F_t) \subseteq Dom\, x_{st}$.
\end{enumerate}
\end{lemma}

\begin{proof}
For all $X \geq 0$, we have 
\[\begin{split}
0 \leq x_{st}(X) &
=  \underset{g \in {\cal C}_{st}^{p}}\essinf\rho_{st}(g-X) - \underset{g \in {\cal C}_{st}^{p}}\essinf\rho_{st}(g)\\
& \leq \rho_{st}(-X) - \underset{g \in {\cal C}_{st}^{p}}\essinf\rho_{st}(g)
\in  L_{p}(\F_s).
\end{split}\]
From the monotonicity, the convexity of $x_{st}$, and $x_{st}(0)=0$, we obtain that, for all $X \in L_{p}(\F_t)$, 
$$
-x_{st}(|X|) \leq -x_{st}(-X) \leq x_{st}(X) \leq x_{st}(|X|).
$$
Then $ |x_{st}(X)| \leq x_{st}(|X|)$ and $x_{st}(X) \in L_{p}(\F_s)$. 
By this we have proved 1.

Now, let $X \in L_\infty({\cal F}_t)$, then there are two real numbers $N,M$ such that $N \leq X \leq M$. It follows easily from the translation invariance property that 	for all $s$, 
$N \leq x_{st}(X) \leq M$. This yields 2.
\end{proof}	

In particular we stress that, for any $p \in [1,\infty]$, the operator $x_{st}$ in (\ref{eqRIP}) restricted to $L_{\infty}({\cal F}_t)$  is always well-defined  with values 
in $L_\infty(\F_s)$.

\bigskip
We give now an alternative formula for $x_{st}$.
\begin{lemma}
For all $s \leq t$, for all $X \in L_{p}({\cal F}_t)$, we have
$$
\underset{g \in {\cal C}^{p}_{st}}{\essinf} \: \rho_{st}(g-X)
\:= \:
\underset{g \in {\cal C}^{\infty}_{st}}{\essinf} \: \rho_{st}(g-X) ,
$$
where 
\[
\mathcal{C}_{st}^{\infty} := \big\{ g \in L_{ \infty} (\F_t) \, : \; \exists\, \theta \in 
\Theta \textrm{ such that } g \leq Y_{st}(\theta) \big\}.
\]
Then, for all $X \in L_{p}({\cal F}_t)$, 
\begin{equation}
x_{st}(X)=\underset{g \in {\cal C}^{\infty}_{st}}{\essinf} \: \rho_{st}(g-X) - 
\underset{g \in {\cal C}^{\infty}_{st}}{\essinf} \: \rho_{st}(g) .
\label{eqRIP2}
\end{equation}	
\label{lemmarest}		
\end{lemma}

\begin{proof}
Since ${\cal C}^{\infty}_{st} \subset {\cal C}^{p}_{st}$, then we have that, for all $X\in L_p(\F_t)$, 
$$
\underset{g \in {\cal C}^{p}_{st}}\essinf  \rho_{st}(g-X) \leq 
\underset{g \in {\cal C}^{\infty}_{st}}\essinf \rho_{st}(g-X) .
$$

Let $ g_0 \in {\cal C}^{p}_{st}$. Let $\theta \in \Theta$  such that  $ g_0 \leq Y_{st}(\theta)$. From Definition \ref{all-strategies} we have that there is  $C>0$ and $Y_{st}(\theta) \geq -C$.
It follows that $g'=\sup(g_0,-C)$ satisfies $g_0 \leq g' \leq Y_{s,t}(\theta)$  and $|g'| \leq \sup (C,|g_0|)$. 
Thus $ g' \in {\cal C}^{p}_{st}$ and $\rho_{st}(g_0 -X) \geq \rho_{st}(g'-X) $.

 The random variable  $g'$ is bounded from below  and thus it is the increasing limit of the sequence $g'_n=\inf(g',n)$. 
 Observe that $g'_n \in L_\infty(\F_t)$, $g'_n \leq Y_{st}(\theta)$.  
The continuity from below of $\rho_{st}$ yields $\rho_{st}(g'-X)=\lim_{n \rightarrow \infty} \rho_{st}(g'_n-X)$. 
This gives
$$
\rho_{st}(g_0-X) \geq \underset{g \in {\cal C}^{\infty}_{st}}\essinf \rho_{st}(g-X) .
$$ 
The proof is complete.
\end{proof}

\subsubsection*{Examples.}
Risk-indifference prices are constructed from fully-dynamic risk measures $(\rho_{st})_{s,t}$ on $(L_p(\F_t))_t$.
As illustration, we have seen earlier those generated from BMO martingales and from BSDEs.
\begin{itemize}
\item
We observe here that \eqref{techass} is always satisfied when we consider a set of strategies $\Theta$ such that, for all $s \leq t$, there is a random variable $\xi_t \in L_p(\F_t)$ such that $Y_{st}(\theta) \leq \xi_t$, for all $\theta \in \Theta$.
This follows from the definition of ${\cal C}^p_{st}$ (Definition \ref{strategiesC}) and the monotonicity of $\rho_{st}$.\\
Also observe that, for these strategies, Lemma \ref{Linf} point 1. holds.
\item
Let us consider the fully-dynamic risk measures from the BMO martingales as constructed in Proposition \ref{BMOcont}.
Let $p'$ be such that $p(K) < p' < p <\infty$, where $p(K)$ is defined in Proposition \ref{BMOcont}.
Consider $\Theta$ to be the set of strategies so that, for all $s\leq t$, there exists a random variable $\xi_t\in L_{p'}(\F_t)$ and $\xi_t \notin L_p(\F_t)$, such that $Y_{st}(\theta) \leq \xi_t$, for all $\theta \in \Theta$.
From the arguments of Lemma \ref{lemmaBMO} and Proposition \ref{BMOcont}, we can see that $\rho_{st}(\xi_t) \in L_{p'}(\F_s)$.
From the monotonicity, we see that \eqref{techass} is satisfied. 
However, in this case, Lemma \ref{Linf} point 1. does not hold.
\end{itemize}

\subsubsection{Fully-dynamic risk measures vs. dynamic risk-measure}
\label{subs2.2.1}

Set $p \in [1,\infty]$. In large part of the literature on risk-measures, a so-called \emph{dynamic risk measure} $(\rho_{s})_{s}$ is characterised by  only one time index. See e.g. \cite{AP}, \cite{KS}. 

In this case the risk measure refers to the evaluation of positions at a fixed time horizon $T<\infty$ and this corresponds to $\rho_s:=\rho_{sT}$. 

In particular, with this definition we have that the risk evaluation of any $X \in L_p(\F_t)$ is the same \emph{for every} $t \in [s,T]$.

In the framework of fully-dynamic risk measures this corresponds to the \emph{restriction property}:
$$
\rho_s(X) =\rho_{sT}(X)= \rho_{st}(X) \quad X\in L_p(\F_t), \quad  s\leq t \leq T.
$$
See \eqref{restriction}.
For the dynamic risk measure $(\rho_t)_t$, \emph{time-consistency} is given as  
  $$
  \rho_t(X)=\rho_t(Y)  \Longrightarrow \rho_s(X)=\rho_s(Y),\qquad s \leq t,\;\; X,Y \in L_p({\cal F}_T).
  $$

 \begin{remark} 
 \label{remark3.1}
From Corollary \ref{Corollary2.5}, we can see that if the fully-dynamic risk measure $(\rho_{st})_{s,t}$ is normalised, i.e. $\rho_{st}(0)=0$ for all $s\leq t$,
then there is a one-to-one correspondence between 
the strong time-consistent $(\rho_{st})_{s,t}$ and 
the time-consistent $(\rho_s)_{s}$.
In fact, 
$$
\rho_{st}(X) =\rho_{sT} (X) = \rho_s (X) ,\qquad X\in L_p(\F_t), \quad  t \geq s.
$$
We stress that the result is not true if $\rho_{st}$ is not normalised.
\end{remark}

\vspace{2mm}
Now, let $(\rho_{st})_{s,t}$ be not normalised, then we know that
the family $(\check\rho_{st})_{s,t}$ with
$$
\check\rho_{st}(X) := \rho_{st}(X+\rho_{st}(0)) = \rho_{st}(X) - \rho_{st}(0), \quad X\in L_p(\F_t)
$$
is normalised.

\begin{remark}
\label{tc-normalised}
If the family $(\rho_{st})_{s,t}$ is strong time-consistent and not normalised, then the normalisation $(\check\rho_{st})_{s,t}$ produces a family which is \emph{not} strong time-consistent, but only time-consistent. See Remark \ref{stc and normalisation} and Remark \ref{tc and normalisation}.
\end{remark}

In the study of risk-indifference pricing, we can see that the definition of the price itself admits equivalent representations in terms of both the fully-dynamic risk-measures $(\rho_{st})_{s,t}$ and $(\check\rho_{st})_{s,t}$.
Namely, we have: 
\begin{align*}
x_{st}(X) &= \underset{g \in {\cal C}_{st}^{p}}{\essinf} \: 
\rho_{st}(g-X) - \underset{g \in {\cal C}_{st}^{p}}{\essinf}\: \rho_{st}(g) \\
&= \underset{g \in {\cal C}_{st}^{p}}{\essinf} \: 
\check\rho_{st}(g-X) - \underset{g \in {\cal C}_{st}^{p}}{\essinf}\: \check\rho_{st}(g), \qquad X\in L_p(\F_t).
\end{align*}
However, the first representation is in terms of a \emph{strong} time-consistent family, while the second is not.

\vspace{2mm}
As we shall see in the sequel of the paper, \emph{strong} time-consistency is \emph{crucial} to guarantee that the risk-indifferent prices $(x_{st})_{s,t}$ constitutes a convex price system. It is then not enough to work with the normalised version of the given strong time-consistent risk-indifferent measure. See Remark \ref{stc and normalisation}.

\vspace{2mm}
For this reason in the sequel of the paper we focus on the general first representation of the risk-indifference prices. See Definition \ref{defRIP}.

\vspace{2mm}
\subsubsection{Properties of the risk-indifference price operator $x_{st}$}

Hereafter we study the properties of the risk-indifference price operator $x_{st}$ for  fixed $s \leq t$. 

\begin{proposition}\label{ri-prop}
The operator $x_{st}(X)$, $X \in L_p(\F_t)$, is monotone, convex, it has the projection property, it is weak $\F_s$-homogeneous, and it is continuous from above.
\end{proposition}

\begin{proof}
The monotonicity of $x_{st}$ follows from the corresponding property of $\rho_{st}$. 
The projection property of 
$x_{st}$ follows from the ${\cal F}_s$-translation invariance of $\rho_{st}$ and $x_{st}(0)=0$.

For the convexity of $x_{st}$ we argue as follows. 
First of all note that ${\cal C}^p_{st}$ is convex, see the convexity of $\Theta$ (Definition \ref{strategies}) and Definition \ref{strategiesC}. 
Therefore, for all $g_1$ and $g_2$ in ${\cal C}^p_{st}$, $\lambda g_1+(1-\lambda) g_2 \in {\cal C}^p_{st}$.
Thus 
$$
\underset{g \in {\cal C}^p_{st}}\inf\rho_{st}(g-[\lambda X+(1-\lambda Y]) \leq \underset{g_1,g_2 \in {\cal C}^p_{st}}\inf\rho_{st}[\lambda (g_1-X)+(1-\lambda) (g_2-Y)].
$$
The convexity of $\rho_{st}$ yields that 
$$\underset{g \in {\cal C}^p_{st}}\inf\rho_{st}(g-[\lambda X+(1-\lambda Y]) \leq \underset{g_1,g_2 \in {\cal C}^p_{st}}\inf[ \lambda \rho_{st}(g_1-X)+(1-\lambda ) \rho_{st}(g_2-Y)].
$$
This gives the convexity of $x_{st}$.

The weak $\F_s$-homogeneity is justified as follows. Let $A \in {\cal F}_s$, then		
\begin{eqnarray}
1_A x_{st}(X)&=&1_A [\underset{g \in {\cal C}_{st}^{p}}{\essinf}  \rho_{st}(g-X) - 
\underset{g \in {\cal C}_{st}^{p}}{\essinf} \rho_{st}(g)]\nonumber\\
&=& \underset{g \in {\cal C}_{st}^{p}}{\essinf}  1_A\rho_{st}(g-X) - 
\underset{g \in {\cal C}_{st}^{p}}{\essinf} 1_A\rho_{st}(g)\nonumber\\
&=& \underset{g \in {\cal C}_{st}^{p}}{\essinf}   1_A\rho_{st}(g-1_AX) - 
\underset{g \in {\cal C}_{st}^{p}}{\essinf} 1_A\rho_{st}(g) \nonumber\\
&=&	1_Ax_{st}(1_AX),
\end{eqnarray}
where the third inequality comes from the weak  ${\cal F}_s$-homogeneity of $\rho_{st}$ as follows:
$$
1_A\rho_{st}(g-X)=1_A\rho_{st}(1_A(g-X))= 1_A\rho_{st}(1_A(g-1_AX))=1_A\rho_{st}(g-1_AX).
$$  
Finally, recall that for every $0 \leq s \leq t \leq T$, the risk measure $\rho_{st}$ is continuous from below.
Let $(X_n)_n$ in $ L_p(\F_t)$ be a sequence decreasing to $X\in L_p(\F_t)$.  
Then we have that $\rho_{st}(g-X)$ is the decreasing limit of
$\rho_{st}(g-X_n)$, for $n\to \infty$. 
Hence,
\[
\begin{split}
x_{st}(X)
&= \underset{g \in {\cal C}_{st}^{p}}{\essinf} \inf_n \rho_{st}(g-X_n) - 
\underset{g \in {\cal C}_{st}^{p}}{\essinf} \rho_{st}(g)\\
& = \inf _n \big[\underset{g \in {\cal C}_{st}^{p}}{\essinf} \rho_{st}(g-X_n) - 
\underset{g \in {\cal C}_{st}^{p}}{\essinf}  \rho_{st}(g)\big]\\
&=\inf_n x_{st}(X_n).
\end{split}
\]
The monotonicity of $x_{st}$ implies that $x_{st}(X)$ is the decreasing limit of  $x_{st}(X_n)$.
The continuity from above of $x_{st}$ is then proved.
 \end{proof}

\vspace{2mm}
In the last part of this subsection we study the Fatou property for the risk-indifference price operator 
$x_{st}$.
For this we shall distinguish the two cases when $p=\infty$ and $p\in [1,\infty)$.  

We have to recall  that
$$
L_\infty(\F_t) \subseteq Dom x_{st} \subseteq L_p(\F_t),\qquad p \geq 1,
$$
from Lemma \ref{Linf} item 2.

\begin{proposition}
\label{Fatouinf}
Let $p = \infty$. The risk-indifference price operator $x_{st}$ admits the following representation
\begin{equation}
\label{eqRCB}
 x_{st}(X)= \underset{Q  \ll P, Q_{|{\cal F}_s} =P}{\esssup}  (E_{Q}(X|{\cal F}_s)-\gamma_{st}(Q)), \qquad  X\in L_{\infty}({\cal F}_t),
 \end{equation}
where $\gamma_{st}(Q)$ is the minimal penalty: 
$$
\gamma_{st}(Q)= \underset{X \in L_{\infty}({\cal F}_t)}{\esssup}  \: (E_{Q}(X|{\cal F}_s)-x_{st}(X)). 
$$
In particular $x_{st}$ has the Fatou property and is continuous on $L_\infty(\F_t)$.
\end{proposition}

\begin{proof}
The representation follows from the continuity from above, see Proposition \ref{ri-prop}. Indeed we refer to the dual representation for conditional risk measures as in \cite{DS2005} or in \cite{BN-preprint}. This dual 
representation is written for the conditional risk measure $x_{st}(-X)$. The second assertion  follows from the representation (\ref{eqRCB}) itself and \cite{DS2005}.
\end{proof}

\begin{remark}\label{Rdomain}
If the operator $x_{st}$ was considered on a Fr\'echet lattice $L_t \supseteq L_\infty(\F_t)$, then one could obtain 
the Fatou property by applying the extension of the Namioka-Klee theorem, see \cite{BF2009}, to the functional $E[x_{st}(X)]$, $X\in L_t$. 

We have to remark that $Dom \, x_{st}$ is not a Fr\'echet lattice, in fact it is not complete and also in general we have $\Theta \subseteq \Xi$ hence it may not even be a vector space.
\end{remark}

As Remark \ref{Rdomain} shows the study of the Fatou property (and also time-consistency as we shall see later) is delicate and the major issues are related to the domain of the operators involved.
This is again noticed in e.g. \cite{FM} in the context of forward utilities.

\bigskip
In view of the remark above, we here propose a different approach to study the Fatou property for $p\in [1,\infty)$.
For all $s,t \in [0,T]: \: s \leq t$, we shall introduce an extension of the operator $x_{st}$ by constructing an 
adequate extension of 
$\rho_{st}$.  
For this, when $p\in [1,\infty)$, we assume that $\rho_{0T}$ is dominated and sensitive.

\bigskip
We introduce the seminorm
$$
c(X) := \sup_{Q\in \mathcal{Q}} E_Q (\vert X \vert ),
$$
for all $\F_T$ measurable random variables $X$
 and
$$\mathcal{Q} :=\{Q \sim P:\; \alpha_{0T}(Q)<\infty\}$$ 
where $\alpha_{0T}$ is the minimal penalty of $\rho_{0T}$ (see Proposition \ref{prop1}).
We observe that
$$
c(X) =0 \quad \iff \quad X=0 \; P-a.s.
$$ 
This follows directly from $Q \sim P$ in the definition of $\mathcal{Q}$.
Then $c$ induces a relationship of equivalence among random variables.

\vspace{1mm}
\begin{defn}\label{L^c}
For all $t$, we define $\mathcal{L}^c_t$ to be the completion, with respect to the seminorm $c$, of the set of essentially bounded $\F_t$-measurable random variables.
We define
$ L^c_t := \mathcal{L}^c_t \big/ \sim$.
\end{defn}

The space $L^c_t$ is a Banach space with norm $c$.

\begin{lem}\label{Lemma L^c}
For all $t$, the following relationship holds for all $Q \in \mathcal{Q}$:
$$
L_p(\F_t, P) \subseteq L^c_t \subseteq L_1(\F_t, Q).
$$
\end{lem}

\begin{proof}
The relationship is directly proved from Proposition \ref{prop1}, in fact $E_Q(\vert X \vert) \leq c(X) \leq K \Vert X \Vert_p$, for all $Q\in \mathcal{Q}$.
\end{proof}

\vspace{2mm}
In the sequel, we extend the risk measure $\rho_{st}$ to obtain the map:
$$
\tilde \rho_{st}: L^c_t \; \longrightarrow \; L^c_s.
$$
With this we can extend the corresponding risk-indifference price $x_{st}$ as a map:
$$
x_{st}: L^c_t \; \longrightarrow \; L^c_s.
$$
This operator allows us to study the Fatou property. 
Also we shall  see that the extensions above are instrumental in the study of time-consistency for the price system.

\vspace{2mm}
In the literature we find other extensions of risk-measures to a larger domain of the type $L_t^ c$. See e.g. \cite{Ref1}, \cite{Ref2}. We observe that these are related to risk-measures with values in $\mathbb{R}$ or $(-\infty,\infty]$.
In this paper, we deal with the extension of the \emph{whole} family of risk measures $(\rho_{st})_{s,t}$ as operators: $\rho_{st}: L_p(\F_t) \longrightarrow L_p(\F_s)$.
We perform the extension of all these operators substantially under the assumptions on the risk measure $\rho_{0T}$ to be dominated and sensitive and the use of the strong time-consistency of the original family $(\rho_{st})_{s,t}$.
We stress that the strong time-consistency is crucial to obtain the extensions $\tilde\rho_{st}$, when $s>0$.
As a result we obtain that the extended family $(\tilde\rho_{st})_{s,t}$ is also strong time-consistent.
In turn this is important to obtain the time-consistency of the risk-indifference prices.
We shall detail these arguments in the sequel.

\subsubsection{Construction of the extended fully-dynamic risk-measure}
The construction of the extension of $\rho_{st}$ is engaging and it requires several steps.
First of all we have the following lemmas.  

\begin{lem}
\label{L1}
For fixed $t \in [0,T]$, consider a convex risk measure $\psi_{tT}: L_\infty(\F_T) \longrightarrow L_\infty(\F_t)$ continuous from below.
Let $\beta_{tT}$ be its minimal penalty. 
Assume that there exists a probability measure $Q \sim P$ such that $E_{Q} (\beta_{tT}(Q)) < \infty$. 
Then the following representation holds:
\begin{equation}
\label{R_t}
\psi_{tT}(X) = \underset{R\sim P ,\, \beta_{tT}(R) \in L_\infty(\F_t)}  \esssup ( E_R(-X\vert \F_t) - \beta_{tT} (R) ), \quad X \in L_\infty (\F_T).
\end{equation}
\end{lem}

\begin{proof}
The proof is organised in steps.

\underline{Step 1}.
Consider $\Phi_t (X) := - \psi_{tT}(X)$, $X \in L_{\infty}(\F_T),\;$  and $\alpha_t(R) := - \beta_{tT}(R)$, $R \sim P$. 
Then the mapping $\Phi_t$ satisfies the 
property {\bf I)} in Theorem 3.1 of \cite{KS}. This result presents a number of equivalent statements and then  property 
{\bf II)} is true.
From the proof that {\bf II)} implies {\bf I)} (see Appendix in \cite{KS}) we have that for all $\varepsilon >0$ there is 
a probability measure $\tilde Q \sim P$ such that
$\alpha_t(\tilde Q) + \varepsilon \geq - \Phi_t(0)$.
On the other hand, we have that
$$ -\Phi_t(0)=\psi_{tT}(0) \geq -\beta_{tT}(\tilde Q)=\alpha_t(\tilde Q).$$
This proves that there exists $\tilde Q \sim P$ such that
 $$
 \beta_{tT}(\tilde Q) = - \alpha_t(\tilde Q)\; \text{belongs to} \;L_\infty(\F_t).
 $$ 
 If follows from the property of the minimal penalty that we can consider $\tilde Q_{\vert \F_t} = P$.
 
 \underline{Step 2}.
Let $X \in L_\infty(\F_t)$. For all $R \sim P$ such that $ R_{\vert \F_t} = P$ let 
$$
A = \big\{ E_R(-X\vert \F_t) - \beta_{tT}(R) > E_{\tilde Q} (-X\vert \F_t) - \beta_{tT}(\tilde Q) \big\}.
$$
Let $\tilde R$ be defined by
$$
\frac{d\tilde R}{dP} = 1_A \frac{d R}{dP} + 1_{A^c} \frac{d\tilde Q}{dP}.
$$
Then we have that
$\tilde R \sim P$, $\tilde R_{\vert \F_t} = P$, and $\beta_{tT} (\tilde R) = 1_A \beta_{tT} (R) + 1_{A^c} \beta_{tT} (\tilde Q)$. 
From the definition of the event $A$, we can see that $\beta_{tT}(R) 1_A \leq \beta_{tT}(\tilde Q) 1_A + 2 \Vert X \Vert_\infty 1_A$.
On the other hand we have $\beta_{tT}(R) \geq - \psi_{tT}(0)$.
Thus we conclude that $\beta_{tT}(\tilde R) \in L_\infty(\F_t)$ and
$E_{\tilde R}\big( X\vert \F_t\big) - \beta_{tT}(\tilde R) \geq E_{ R}\big( X\vert \F_t\big) - \beta_{tT}( R)$.
By this we have obtained representation \eqref{R_t}.
\end{proof}

\begin{lem}
\label{L2}
Set $p\in [1,\infty)$.
Let $(\rho_{st})_{s,t}$ be a  strong time-consistent fully-dynamic risk measure such that $\rho_{0T}$ is dominated and sensitive.
For all $Q \in \mathcal{Q}$, for all  $s \in [0,T]$, we have that the minimal penalties $(\alpha_{st})_{s,t}$ satisfy 
the following:
$$
\alpha_{0s}(Q) \in \R, \;\text{and}\;
\quad E_Q(\alpha_{sT}(Q)) \in\R.
$$
\end{lem}

\begin{proof}
The minimal penalties $(\alpha_{st})_{s,t}$ satisfy the cocycle condition:
\begin{equation}
\label{cc}
\alpha_{0T}(Q) = \alpha_{0s}(Q) + E_Q(\alpha_{sT}(Q)), \quad Q \in \mathcal{Q},
\end{equation}
see \cite{BN02} (Remark \ref{charcoc}).
Furthermore, it follows from the definition of  minimal penalty that
\[\begin{split}
& \alpha_{0s}(Q) \geq - \rho_{0s}(0)\\
& \alpha_{sT}(Q) \geq - \rho_{sT}(0).
\end{split}\]
By hypothesis we have that $\rho_{sT}(0) \in L_p(\F_s)$ and, from Proposition \ref{prop1}, we have that 
$\frac{dQ}{dP} \in L_q(\F_T)$, with $q=p(p-1)^{-1} $.
Thus $E_Q(-\rho_{sT}(0) ) \in \mathbb{R}$. 
The result  follows from  $\alpha_{0T}(Q) < \infty$ and the cocycle condition \eqref{cc}.
\end{proof}

\bigskip
We introduce the following sets of probability measures for all $s\leq t$ on $(\Omega,\F_t)$:
$$
\mathcal{P}_{st} := \Big\{ R \sim P\,: \; R_{\vert \F_s} = P \textrm{ and } \alpha_{st}(R) \in L_p(\F_s) \Big\}
$$
and
$$
\widetilde{\mathcal{P}}_{st} := \Big\{ R \sim P\,: \; R_{\vert \F_s} = P \textrm{ and } 
\sup_{Q\in \mathcal{Q}} E_Q(\alpha_{st}(R)) < \infty \Big\}
$$
We also remark immediately that 
\begin{equation}
 \label{eqNUM}
\Pp  \subseteq \wPp = \Big\{ R \sim P\,: \; R_{\vert \F_s} = P \textrm{ and } 
\sup_{Q\in \mathcal{Q}} E_Q(\vert \alpha_{st}(R) \vert) < \infty \Big\}  .
\end{equation}

In fact, for $R\in \wPp$ ,we have that $\alpha_{st}(R) \geq - \rho_{st}(0)$.
Then $\vert \alpha_{st}(R) \vert \leq \alpha_{st}(R) + 2 \vert \rho_{st}(0) \vert$. We remind that 
$ \rho_{st}(0) \in L_p(\F_t)$ and 
$||\frac{dQ}{dP}||_ {L_q} \leq K$ 
for all $Q \in {\cal Q}$, see Proposition \ref{prop1}.
Then we conclude that both relations hold.

\begin{proposition}
\label{Prop2rep}
Set $p\in [1,\infty)$.
Let $(\rho_{st})_{s,t}$ be a strong time-consistent fully-dynamic risk measure such that $\rho_{0T}$ is dominated and  sensitive.
Then we have that
\begin{enumerate}
\item
for all $t\in [0,T]$, the risk measure $\rho_{0t}$ is dominated and  sensitive,
\item
for all  $t\in [0,T]: \, s \leq t$, the following representation holds:
\begin{equation}
\label{Rep2}
\rho_{st}(X) = \underset{R\in \Pp}\esssup \big( E_R(-X\vert \F_s ) - \alpha_{st}(R) \big), \quad X\in L_p(\F_t),
\end{equation}
\item
for all $Q\in \mathcal{Q}$ and $R\in \wPp$, there exists $h \in L_q(\F_t):$ $\Vert h \Vert_q \leq K$, with $q=p(p-1)^{-1}$, 
such that
$$
E_Q( E_R (X\vert \F_s)) = E(hX), \quad X \in L_p(\F_t),
$$
\item
for all  $t\in [0,T]: \, s \leq t$, the following representation holds:
\begin{equation}
\label{Rep3}
\rho_{st}(X) = \underset{R\in \wPp}\esssup \big( E_R(-X\vert \F_s ) - \alpha_{st}(R) \big), \quad X\in L_p(\F_t).
\end{equation}
\end{enumerate}
\end{proposition}

\begin{proof}\hspace{1cm}
\begin{enumerate}
\item[1.]
For all $Z \in L_p(\F_t)$ we have $\rho_{0T}(Z) = \rho_{0t}(-\rho_{tT}(Z)) = \rho_{0t}(Z-\rho_{tT}(0))$.
By hypothesis, $\rho_{tT}(0) \in L_p(\F_t)$, thus for all $Y \in L_p(\F_t)$ we have
$\rho_{0t}(Y) = \rho_{0T}(Y+\rho_{tT}(0))$.
This shows that $\rho_{0t}$ is dominated. From Lemma \ref{L2}, $\rho_{0t}$ is sensitive.
This completes the proof of item 1.

\item[2.] and 3. The proofs of items 2 and 3 proceed together, first proving the result in item 2  for $X\in L_\infty(\F_t)$, then item 3, and finally, item 2 for $X\in L_p(\F_t)$. The argument is
split in steps.

\bigskip
\underline{Step 1}.
Let $0 \leq s \leq t\leq T$. Define
\begin{equation}
\label{eq rho bar}
\breve\rho_{st} (X) := \rho_{st}(X) - \rho_{st}(0), \qquad X\in L_p(\F_t).
\end{equation}
The translation invariance and the monotonicity of $\rho_{st}$ imply that
$$
\breve\rho_{st}:  L_\infty(\F_t) \longrightarrow L_\infty(\F_s).
$$
The minimal penalty  associated to the restriction of $\breve \rho_{st}$ to $L_{\infty}(\F_t)$ is
$$
\breve\alpha_{st} (Q) = \underset{X\in L_{\infty}(\F_t)}\esssup \big( E_Q(-X\vert \F_s) - \breve\rho_{st}(X) \big).
$$ 
for all $Q\sim P$. Thus
$$
\breve\alpha_{st} (Q) = \underset{X\in L_{\infty}(\F_t)}\esssup  \big( E_Q(-X\vert \F_s) - \rho_{st}(X) \big) + \rho_{st}(0).
$$ 
Then Lemma \ref{lemmaA} yields
\begin{equation}
\label{eq alpha}
\breve\alpha_{st}(Q) = \alpha_{st}(Q) + \rho_{st}(0).
\end{equation}
From (\ref{eq alpha}) and Lemma \ref{L2} we obtain that for all $Q \in {\cal Q}$,  $E_Q(\breve\alpha_{st}(Q))<\infty$ for all $Q$ such that $E (\alpha_{st}(Q)) < \infty$.
From Lemma \ref{L1} we have that the risk measure $\breve \rho_{st}$ satisfies the representation \eqref{R_t}.
This, together with \eqref{eq rho bar} and \eqref{eq alpha} and $\rho_{st}(0) \in L_p({\cal F}_s)$,  proves that \eqref{Rep2} is satisfied for all $X \in L_\infty(\F_t)$.

\underline{Step 2}.
Let $Q\in \mathcal{Q}$ and $R\in \mathcal{ \tilde P}_{st}$.
For all $X \in L_p(\F_t)$ we have 
$$
E_Q(E_R(-X\vert\F_s)) - \alpha_{0s} (Q) - E_Q(\alpha_{st} (R)) \leq \rho_{0s}(-\rho_{st}(X)) = \rho_{0t}(X)$$
From Step 1, $\rho_{0t}$ is dominated: $\rho_{0t}(X)\leq K \Vert X \Vert_p+  C
$. 
Moreover, $\alpha_{0s}(Q) < \infty$ and $\sup_{Q \in \mathcal{Q}} E_Q( \alpha_{st}(R)) < \infty$.
Thus there is some $\kappa \in \mathbb{R}$ such that
$$
E_Q(E_R(-X\vert\F_s)) \leq K \Vert X \Vert_p +  \kappa, \qquad X \in L_p(\F_t).
$$
This proves item 3.

\underline{Step 3}.
Let $Q \in \mathcal{Q}$. 
The functional $\phi(X) := E_Q(\rho_{st}(X))$, $X\in L_p(\F_t)$, defines a finite risk measure on $L_p(\F_t)$.
It is thus continuous for the $L_p$-norm, see \cite{FiSv}.
From the definition of the minimal penalty we have that
$$
\rho_{st} (X) \geq \underset{R\in \mathcal{P}_{st}}\esssup \Big(  E_R (-X\vert \F_s) - \alpha_{st}(R) \Big).
$$
Furthermore, for any given $X$, the set
$$
\Big\{ E_R(-X\vert \F_s) - \alpha_{st}(R), \: R\in \mathcal{P}_{st} \big\}
$$
is a lattice upward directed.
Thus, in view of Proposition VI1.1 in \cite{Neveu}, in order to prove \eqref{Rep2} for all $X \in L_p(\F_t)$, it is enough to prove that
\begin{align}
\label{E}
\phi(X):&= E_Q(\rho_{st}(X)) \\&
= \sup_{R \in \mathcal{P}_{st}} \Big( E_Q(E_R(-X\vert \F_s) - E_Q(\alpha_{st}(R) \Big),
\quad X \in L_p(\F_t).\notag
\end{align}
From Step 1,  we already know that \eqref{E} is satisfied for all $X \in L_\infty (\F_t)$.
Observe that $\phi$ is continuous in the $L_p$-norm. Moreover, the continuity of the right-hand side of \eqref{E} in the $L_p$-norm follows from Step 2 and the inclusion $\mathcal{P}_{st} \subseteq  \mathcal{\tilde P}_{st}$. 
This ends the proof of item 2.
\item[4.]
The representation \eqref{Rep3} follows directly from \eqref{Rep2} and the observation that $\mathcal{P}_{st} \subseteq  \mathcal{\tilde P}_{st}$.
\end{enumerate}
By this the proof is complete.
\end{proof}

\bigskip
\begin{lemma}\hspace{1cm}
\begin{enumerate}
 \item  For all $0 \leq r \leq s \leq t \leq T$,  for all $R \in \tilde{\cal P}_{st}$, there is $S \in \tilde{\cal P}_{sT}$, such that $R$ is the restriction of $S$ to ${\cal F}_t$.
 \item For all $0 \leq r \leq s \leq t \leq T$,   $\tilde{\cal P}_{rt} \subset \tilde{\cal P}_{rs}$, which means that the restriction to ${\cal F}_s$ of an element of $\tilde{\cal P}_{rt}$, belongs to $\tilde{\cal P}_{rs}$.
\end{enumerate}
\label{LemmaER}
\end{lemma}

\begin{proof}\hspace{1cm}
 \begin{enumerate}
  \item 
  $R \in \tilde{\cal P}_{st}$, let $\tilde Q \in {\cal P}_{tT}$. Let $S$ be the probability measure equivalent with $P$ such that 
  $$\frac{dS}{dP}=\frac{dR}{dP}   \frac{\frac{d\tilde{Q}}{dP}}{E(\frac{d\tilde{Q}}{dP}|{\cal F}_t)}.$$
  In particular the restriction of $S$ to ${\cal F}_t$ is equal to $R$. It follows from the definition of $S$ and the properties of the minimal penalty that 
  $$\alpha_{sT}(S)=\alpha_{st}(R)+E_R(\alpha_{tT}(\tilde{Q})|{\cal F}_s).$$
  Thus 
  \begin{equation}
  \sup_{Q \in {\cal Q}}E_Q(\alpha_{sT}(S))\leq \sup_{Q \in {\cal Q}}E_Q(\alpha_{st}(R))+ 
  \sup_{Q \in {\cal Q}}E_Q[E_R(\alpha_{tT}(\tilde{Q}|{\cal F}_s)].
  \label{eqPst}
  \end{equation}
  From item 3 of Proposition \ref{Prop2rep}, we have that, for every $Q \in {\cal Q}$, there is $h \in L_q({\cal F}_t)$ with $||h||_q \leq K$ such that, for all $X$ in $L_p({\cal F}_t)$, $E_Q[E_R(X|{\cal F}_s)]=E(hX) \leq K ||X||_p$. 
  By definition of $\tilde{\cal P}_{st}$ and 
  ${\cal P}_{t,T}$ and using (\ref{eqPst}), we have that $S \in \tilde{\cal P}_{sT}$. 
  \item 
  Let $R \in \tilde{\cal P}_{rt}$. It follows from the definition of the minimal penalty that $\alpha_{st}(R) 
  \geq -\rho_{st}(0)$. We deduce then from the cocycle condition that 
  $$\alpha_{rs}(R) \leq \alpha_{rt}(R)+E_R(\rho_{st}(0)|{\cal F}_r)$$
  where $\rho_{st}(0)$ belongs to $L_p({\cal F}_s)$. 
  Then, from Proposition \ref{Prop2rep} item 3, we obtain $\sup_{Q \in {\cal Q}} E_Q(\alpha_{rs}(R)) <\infty$. This gives the result.
 \end{enumerate}
 The proof is complete.
 \end{proof}

\bigskip
We  will now extend $\rho_{st}$ to $L^c_t$ for every $0 \leq s \leq t \leq T$. We first prove the following result.

\begin{proposition}
 Let $0=s_0 <s_1<...<s_n=T$.
 For all $Q_i \in \tilde{\cal P}_{s_is_{i+1}}$ with $i=0,...,n-1$, 
 let $Q$ be the unique probability measure  on ${\cal F}_T$ such that 
 $$
 E_Q(X)=E_{Q_0}(E_{Q_1}( \cdots E_{Q_{n-1}}(X|{\cal F}_{s_{n-1}}) \cdots |{\cal F}_{s_1})),
 \quad X\in L_{\infty}({\cal F}_T).
 $$
 Then $Q$ belongs to ${\cal Q}$.
 \label{propcomp}
\end{proposition}

\begin{proof}
For $n=1$ observe that $\tilde{\cal P}_{0T}={\cal Q}$. 
We prove the result by induction for $n \geq 2$.
 
\underline{Step 1}. 
For $n=2$, we have $0=s_0 <s_1< s_2=T$ and $E_Q(X)=E_{Q_0}(E_{Q_1}(X|{\cal F}_{s_1}))$, for any $X\in L_\infty (\F_T)$.  
From the cocycle condition and the properties of the minimal penalty, it is
$$
\alpha_{0T}(Q)=\alpha_{0s_1}(Q_0)+E_{Q_0}(\alpha_{s_1T}(Q_1)).
$$
By assumption the probability measure $Q_0 \in \tilde{\cal P}_{0s_1}$. 
From item 1 of Lemma \ref{LemmaER} we can see that $Q_0$ is the restriction to ${\cal F}_{s_1}$  of an element of  ${\cal Q}$. Then we easily see also that $\alpha_{0T}(Q)<\infty$ from the definition of $\tilde {\cal P}_{s_1T}$. Thus $Q\in {\cal Q}$.
 
 \underline{Step 2}.
 We assume that the result holds for $n$ and we prove it for $n+1$.
 Let $0= s_0 <s_1<...<s_n<s_{n+1}=T$.
 From the induction hypothesis $Q_{n-1} \in \tilde{\cal P}_{s_{n-1}s_n}$. 
 From item 1 of Lemma \ref{LemmaER} we obtain that $Q_{n-1}$ is the restriction to 
 ${\cal F}_{s_n}$ of an element $R_{n-1} \in \tilde{\cal P}_{s_{n-1}T}$. 
Then by induction we can see that the probability measure $R$ on ${\cal F}_T$ defined by
 $$
 E_R(X)=E_{Q_0}(E_{Q_1}(\cdots E_{R_{n-1}}(X|{\cal F}_{s_{n-1}})\cdots |{\cal F}_{s_1}))
 $$
 for all  $X$ in $L_{\infty}({\cal F}_T)$, belongs to ${\cal Q}$. 
 From item 2 of Lemma \ref{LemmaER}, the restriction $\tilde R$ of $R$ to ${\cal F}_{s_n}$ belongs to $\tilde{\cal P}_{0s_n}$ and
 $$
 E_Q(X)=E_{Q_0}(E_{Q_1}(\cdots E_{Q_{n}}(X|{\cal F}_{s_{n}}) \cdots|{\cal F}_{s_1})=E_{\tilde R}(E_{Q_{n}}(X|{\cal F}_{s_{n}}).
 $$ 
 The result follows then from Step 1.
\end{proof}

\bigskip
\begin{theorem}
Set $p\in [1,\infty)$.
Let $(\rho_{st})_{s,t}$ be a  strong time-consistent fully-dynamic risk measure on $(L_p({\cal F}_t))_t$ such that $\rho_{0T}$ 
is dominated and sensitive.
\begin{enumerate}
 \item 
 For all $s \leq t $, the risk measure $\rho_{st}$ admits a unique extension $\tilde \rho_{st}:L^c_t \rightarrow L^c_s$ such 
 that  
 \begin{equation}
 c(|\tilde \rho_{st}(X)-\tilde \rho_{st}(Y)|) \leq c(|X-Y|), \qquad X,Y \in L^c_t.
 \label{eqcrho}
 \end{equation}
 \item The extension $\tilde \rho_{st}$ admits the following representation
 \begin{equation}
 \tilde \rho_{st}(X)=\underset{R \in \tilde {\cal P}_{st}}\esssup (E_R(-X|{\cal F}_s)-\alpha_{s,t}(R)), \qquad X \in L^c_t.
  \label{Rep4}
 \end{equation}
\end{enumerate}
\label{thmextrho} 
\end{theorem}

\begin{proof}\hspace{1cm}
\begin{enumerate}
\item
Let $X,Y \in L_{\infty}({\cal F}_t)$. 
Consider $A \in {\cal F}_s$ such that $|\rho_{st}(X)-\rho_{st}(Y)|=[\rho_{st}(X)-\rho_{st}(Y)]1_A+[\rho_{st}(Y)-
 \rho_{st}(X)]1_{A^c}$. From (\ref{Rep3}) we get
 \begin{equation}
 |\rho_{st}(X)-\rho_{st}(Y)| \leq \underset{R \in \tilde {\cal P}_{st}}\esssup \;E_R(|X-Y||{\cal F}_s).
  \label{eqrho1}
  \end{equation}
For all $X,Y \in L_{\infty}({\cal F}_t)$, the set $\{E_R(|X-Y||{\cal F}_s),R \in \tilde {\cal P}_{st}\}$ satisfies the lattice property. It follows that, for all $Q \in {\cal Q}$,
$$
E_Q \Big(\underset{R \in \tilde {\cal P}_{st}}\esssup E_R(|X-Y||{\cal F}_s)\Big)
=\underset{R \in \tilde {\cal P}_{st}}\sup E_Q \Big(E_R(|X-Y||{\cal F}_s)\Big).
$$
From equation (\ref{eqrho1}), from item 1 of Lemma \ref{LemmaER}, and from Proposition \ref{propcomp}, we obtain that
\begin{equation}
  c(|\rho_{st}(X)-\rho_{st}(Y)|) \leq c(|X-Y|),\qquad  X,Y \in L_{\infty}({\cal F}_t).
  \label{eqccont0}
  \end{equation}
On the other hand, $\rho_{st}(X) \in L_p({\cal F}_s) \subset L^c_s$, for all $X \in L_{\infty}({\cal F}_t)$.
Then item 1 in the statement of the theorem follows from the density of $L_{\infty}({\cal F}_t)$ 
 in $L^c_t$ and from equation (\ref{eqccont0}).
 \item
 The representation (\ref{Rep4}) is satisfied for all $X \in L_{\infty}({\cal F}_t)$, see item 4 of  Proposition \ref{Prop2rep}.
 Item 1 of the theorem provides the continuity of $\tilde \rho_{st}$ for the c norm.
Moreover, from Lemma \ref{LemmaER} item 1 
 and Proposition \ref{propcomp}, we have continuity for the c norm of the right-hand side of \eqref{Rep4}. By continuity, we can then see that (\ref{Rep4}) is satisfied for all $X$ in $L^c_t$.
 \end{enumerate}
 This completes the proof.
\end{proof}

\begin{corollary}
\label{cortildero}
The family  $(\tilde \rho_{st})_{s, t}$ is a strong time-consistent dynamic risk measure on $(L^c_t)$.  Furthermore $\tilde \rho_{0T}$ is dominated by 
 $\sup_{Q \in {\cal Q}} E_Q(-X)$.
\end{corollary}

\begin{proof}
The properties of monotonicity, translation invariance, and strong time-consistency for 
$(\tilde\rho_{st})_{s,t}$ follow from (\ref{eqcrho}) and the corresponding properties 
for $(\rho_{st})_{s,t}$ on $(L_p({\cal F}_t))_t$. The domination by $\sup_{Q \in {\cal Q}} E_Q(-X)$ follows from the representation 
(\ref{Rep4}) of $\tilde\rho_{0T}$ and the observation that $\tilde {\cal P}_{0T}={\cal Q}.$
\end{proof}

\subsubsection{The extended risk-indifference prices}

Now that we have the extension $(\tilde\rho_{st})_{s,t}$ on $(L^c_t)_t$, we can extend the price system $(x_{st})_{s,t}$ on $(L^c_t)_t$.
Note that we shall keep the same notation $x_{st}$ also for the corresponding extended price operator. 
 
 \begin{defn}
 Set $p\in [1,\infty)$.
Let $(\rho_{st})_{s,t}$ be a strong time-consistent, fully-dynamic risk measure on $(L_p({\cal F}_t))_t$ such that $\rho_{0T}$ is 
dominated and  sensitive.
  Let $s \leq t$. For all $X \in L^c_t$, define 
  $$x_{st}(X):=\underset{g \in {\cal C}_{st}^{\infty}}{\essinf}  \tilde \rho_{st}(g-X) - 
  \underset{g \in {\cal C}_{st}^{\infty}}{\essinf} \tilde \rho_{st}(g),$$
where $\tilde \rho_{st}$ is the extension of $\rho_{st}$ to $L^c_t$ defined in Theorem \ref{thmextrho} .
  \label{defextx}
 \end{defn}

\begin{proposition}\hspace{1cm}
\begin{enumerate}
 \item 
 The operator $x_{st}$ is well-defined on $L^c_t$ with values in $L^c_s$. 
 It extends the operator in Definition \ref{defRIP} (see also \eqref{eqRIP2}).
 Moreover, 
 \begin{equation}
 c(|x_{st}(X)-x_{st}(Y)|) \leq c(|X-Y|), \qquad X,Y \in L^c_t.
  \label{eqceqcontx}
 \end{equation}
\item 
The operator $x_{st}$ is convex, monotone, and satisfies the projection property.
\item 
Choose $Q_0 \in {\cal Q}$, then $x_{st}$ admits the following dual representation:
\begin{equation}
x_{st}(X)=\underset{R \in {\cal K}}{\esssup} (E_R(X|{\cal F}_s)-\gamma_{st}(R)),\qquad X \in L^c_t,
\label{eqrepx}
\end{equation}
where ${\cal K}$ is a set of probability measures in the dual of $L^c_t$ compact for the weak* topolology such that every $R \in {\cal K}$ is absolutely continuous with respect to $P$,
such that the restriction 
of $R$ to ${\cal F}_s$ is equal to $Q_0$, and 
$$
\gamma_{st}(R)= \underset{Y \in L^c_t}\esssup \big(E_R(Y|{\cal F}_s)-x_{st}(Y)\big).
$$
Furthermore, for all $X \in L^c_t$, there exists $Q_X \in {\cal K}$ such that 
\begin{equation}
x_{st}(X)=E_{Q_X}(X|{\cal F}_s)-\gamma_{st}(Q_X)).
\label{eqrepx2}
\end{equation}
\end{enumerate}
\label{propxL}
\end{proposition}

\begin{proof}\hspace{1cm}
\begin{enumerate}
\item
From (\ref{Rep4}) in Theorem  \ref{thmextrho} we have that 
 $$
 |x_{st}(X)-x_{st}(Y)| \leq \underset{R \in \tilde {\cal P}_{st}}\esssup E_R(|X-Y||{\cal F}_s),
\quad X,Y \in L^c_t. 
 $$
Then equation (\ref{eqceqcontx}) follows from Lemma \ref{LemmaER} item 1 and Proposition \ref{propcomp}.

Recall that $x_{st}(X) \in L_{\infty}({\cal F}_s) \subset L^c_s$, for all $X \in L_{\infty}({\cal F}_t)$. From the continuity of $x_{st}$ with respect to the c norm we conclude that $x_{st}(X) \in L^c_s$, for all $X \in L^c_t$.

Finally, Definition \ref{defextx} of $x_{st}$ extends the one  given in Definition \ref{defRIP} (see also \eqref{eqRIP2}), because $\tilde \rho_{st}$ is the extension of $\rho_{st}$ to $L^c_t$.
\item
Statement 2 follows directly from the properties of $\tilde \rho_{st}$, see Corollary \ref{cortildero}, in the same lines as Proposition \ref{ri-prop}.
\item
Let $Q_0 \in {\cal Q}$. 
The map $E_{Q_0}(x_{st}(X))$, $X \in L^c_t$, is up to a minus sign a normalised convex risk measure on $L^c_t$ with values in $\R$ majorized by $\sup_{Q \in {\cal Q}} E_Q(X)$.
We refer to Proposition 3.1 and Theorem 3.2 in \cite{BNK} to prove the existence of a set ${\cal K}$ of probability measures in the dual of $L^c_t$, compact for the weak* topology, such that 
 \begin{eqnarray*}
  E_{Q_0}(x_{st}(X))&=&\sup_{R \in {\cal K}} \,(E_R(X)-\gamma(R))\nonumber\\
  &=& E_{Q_X}(X)-\gamma(Q_X)
 \end{eqnarray*}
for some $Q_X \in {\cal K}$ (depending on $X$). 
Thus the representations (\ref{eqrepx}) and (\ref{eqrepx2}) follow from standard arguments, 
see e.g. \cite{DS2005}.
\end{enumerate}
The proof is complete.
\end{proof}

\bigskip
Now we are ready to discuss the Fatou property of a risk-indifferent evaluation, as given in Definition \ref{defextx}. 
This brings to the following natural definition.

\begin{defn}
For any $s \leq t$, an operator $x:L^c_t \rightarrow L^c_s$ has \emph{the Fatou property on $L^c_t$} if for any sequence $(X_n)_n \in L^c_t$, dominated in $L^c_t$, and converging $P$-a.s. to $X  \in  L^c_t$, we have
 \begin{equation}
 x(X) \leq \underset{n \rightarrow \infty}\liminf \;x(X_n).
  \label{eqFatou}
 \end{equation}
Here above $(X_n)_n$ \emph{dominated in $L^c_t$} means that there is $Y \in L^c_t$ such that $|X_n| \leq Y$  $P$-a.s. for all $n$.
\end{defn}

\begin{proposition}
Set $p\in [1,\infty)$.
Let $(\rho_{st})_{s,t}$ be a  strong time-consistent, fully-dynamic risk measure on $(L_p({\cal F}_t))_t$ such that $\rho_{0T}$ is 
dominated and  sensitive.
For all $s \leq t$, the operator $x_{st}$ in Definition \ref{defextx} has the Fatou property on $L^c_t$.
 \label{propFatou}
\end{proposition}

\begin{proof}
Let $(X_k)_k \in L^c_t$ be dominated in $L^c_t$ by $Y$ and converge $P$-a.s. to $X\in L^c_t$. 
Set  $\tilde X_n:=\underset{k \geq n}\inf X_k$.
The sequence $(\tilde X_n)_n$ is increasing, $|\tilde X_n| \leq Y $, and 
 $X=\underset{n \rightarrow \infty}\lim \tilde X_n$ $P$-a.s.
 From Proposition \ref{propxL}, for all $X\in L^c_t$, there is a probability measure $Q_X$ in the dual of $L^c_t$  such that $x_{st}(X)=E_{Q_X}(X|{\cal F}_s)-\gamma_{st}(Q_X)$.
Then, by the dominated convergence theorem, we obtain that 
 \begin{align}
 x_{st}(X)&= \underset{n \rightarrow \infty}\lim \Big(E_{Q_X}(\tilde X_n|{\cal F}_s)-\gamma_{st}(Q_X)\Big) \notag \\
 &\leq 
 \underset{n \rightarrow \infty}\lim \underset{k \geq n}\inf \Big(E_{Q_X}( X_k|{\cal F}_s)-\gamma_{st}(Q_X)\Big).
 \label{eqFat1}
 \end{align}
 Observe that, for all $k$, $E_{Q_X}( X_k|{\cal F}_s)-\gamma_{st}(Q_X) \leq x_{st}(X_k)$.
Then from (\ref{eqFat1}) we get that
 $$x_{st}(X) \leq \underset{n \rightarrow \infty}\lim \underset{k \geq n}\inf x_{st}(X_k)=
 \underset{n \rightarrow \infty}\liminf x_{st}(X_n).$$
 This proves the Fatou property.
\end{proof}

\bigskip
\subsection{Risk-indifference price system $(x_{st})_{s,t}$}

At this stage we have studied the properties of $x_{st}$ defined on $L^c_t$ with values in $L^c_s$. 
We now study the time-consistency of the family $(x_{st})_{s,t}$.

\begin{lemma}
 Let $0 \leq r \leq s\leq t$. 
 Every  $g \in \mathcal{C}^{\infty}_{rt}$ can be written $g=g_1+g_2$ for some $g_1 \in \mathcal{C}^{\infty}_{rs}$ and $g_2 \in \mathcal{C}^{\infty}_{st}$ .
\label{lemmasumC}
 \end{lemma}
 
\begin{proof}
 Let $\theta \in \Theta$ such that $g \leq Y_{rt}(\theta)=Y_{rs}(\theta)+Y_{st}(\theta)$. 
 Let $M \in \R$ such that $M \leq g$, $M \leq Y_{rs}(\theta)$,  $M \leq Y_{st}(\theta)$. 
 Let 
 $$
 g_1 := \inf \big(Y_{rs}(\theta),||g||_{\infty}-M\big).
 $$
 We remark that $g_1 \in L_{\infty}({\cal F}_s)$ and $g_1 \leq Y_{rs}(\theta)$. 
 Let $g_2 := g-g_1$, clearly $g_2 \in L_{\infty}({\cal F}_t)$.
Then we can see that $g_2 = \sup \big(g-Y_{rs}(\theta),g-||g||_{\infty}+M\big)$. 
Hence, we obtain that 
 $g_2\leq \sup(Y_{st}(\theta), M)=Y_{st}(\theta)$.
\end{proof}

\begin{lemma}
\label{propwtc}
Let $p \in [1,\infty]$. 
Let $(\rho_{st})_{s,t}$ be a strong time-consistent fully-dynamic risk measure on $(L_p({\cal F}_t))_t$.
Let $r \leq s\leq t$. 
Consider the operator $x_{st}(X)$, $X\in L_p(\F_t)$ defined $P$-a.s. as in \eqref{eqRIP}.
For any $X,Y \in L_p({\cal F}_t)$, assume that $x_{st}(X) \geq x_{st}(Y)$ $P-a.s.$ 
Then $x_{rt}(X)  \geq x_{rt}(Y)$.
\end{lemma}
Observe that the operator in Lemma \ref{propwtc} would be a risk-indifference price if defined on $Dom \, x_{st}$, see Definition \ref{defRIP}.

\begin{proof}
For the proof we deal with both $L_p$-convergence and $P-a.s.$ convergence. For this reason we work with an extension of the risk measure 
$\rho_{rs}$ from $L_p(\F_s)$ to the set
$$
D_s:= \{ \F_s- \textrm{measurable } X\,: \exists (X_n)_n \in L_p(\F_s) \,s.t.\, X_n \uparrow X \:P-a.s \}.
$$
We define this extension as the monotone limit
$$
\bar \rho_{rs}(X) := \lim_{n \to \infty} \rho_{rs} (X_n) \quad P-a.s.
$$
First of all we show that the definition is well-posed. 
We consider two sequences $(X_n)_n$ and $(Y_n)_n$ in $L_p(\F_s)$ such that both $X_n \uparrow X$ and $Y_n \uparrow X$ $P-a.s.$, where $X$ is $\F_s$-measurable. 
We denote $A := \lim_{n\to\infty} \rho_{rs}(X_n)$ and $B:= \lim_{n\to \infty} \rho_{rs}(Y_n)$ $P-a.s.$
Then for any constant $M$, we would have
$X_n \wedge M  \uparrow X \wedge M$ and also $Y_n \wedge M  \uparrow X\wedge M$ in the $L_p$-convergence.
Observe that $X\wedge M \in L_p(\F_s)$ and $\rho_{rs}$ is continuous from below, thus we have
$$
\inf_M \inf_n \rho_{rs} (X_n\wedge M) = \inf_M  \rho_{rs} (X\wedge M) = \inf_M \inf_n \rho_{rs} (Y_n\wedge M).
$$
Hence $A= \bar\rho_{rs}(X) = B$.  
It is also clear that $\bar\rho_{rs}$ is monotone and continuous $P$-a.s. from below on $D$.

We now proceed with the proof of the statement.
Let $X,Y \in L_p({\cal F}_t)$ such that $x_{st}(X) \geq x_{st}(Y)$.
From (\ref{eqRIP2}) and assumption (\ref{techass}),  we have that 
\begin{equation}
\underset{g \in {\cal C}^{\infty}_{st}}{\essinf}\: \rho_{st}(g-X) \geq 
\underset{g \in {\cal C}^{\infty}_{st}}{\essinf}\: \rho_{st}(g-Y).
\label{eqwtc1}
\end{equation}
From Lemma \ref{lemmasumC}, any $g \in {\cal C}^{\infty}_{rt}$ is the sum $g= g_1+g_2$ of $g_1\in {\cal C}^{\infty}_{rs}$ and 
$g_2\in {\cal C}^{\infty}_{st}$. 
First of all from \eqref{eqwtc1} and the lattice property of $\{\rho_{st}(g-X),\; g \in {\cal C}^{\infty}_{st}\}$ we observe
that $g_1-\underset{g_2 \in {\cal C}^{\infty}_{st}}{\essinf} \:\rho_{st}(g_2-X)$ belongs to $D_s$. 
Then we can apply the extension $\bar\rho_{rs}$ of $\rho_{rs}$ and from its monotonicity we obtain that 
\begin{equation}
\label{eqstar1}
\bar  \rho_{rs}(g_1- \underset{g_2 \in {\cal C}^{\infty}_{st}}{\essinf} \:\rho_{st}(g_2-X)
 \geq 
 \bar \rho_{rs}(g_1- \underset{g_2 \in {\cal C}^{\infty}_{st}}{\essinf} \:\rho_{st}(g_2-Y), 
\end{equation}
for all $g_1$ in ${\cal C}^{\infty}_{rs}$.

On the other hand the strong time-consistency of $(\rho_{st})_{s,t}$  yields
\begin{equation}
\underset{g \in {\cal C}^{\infty}_{rt}}{\essinf}  \rho_{rt}(g-X) 
= \underset{\substack{{g_1 \in {\cal C}^{\infty}_{rs} }\\ 
{g_2 \in {\cal C}^{\infty}_{st}} }}{\essinf} \:\rho_{rs}(g_1- \rho_{st}(g_2-X)).
\label{eqg1g2}
\end{equation}
Now, let us consider a sequence $(h_n)_n$ such that $-  \essinf_{g \in {\cal C}^\infty_{st}} \rho_{st}(g-X)$ is the 
increasing limit of $- \rho_{st}(h_n-X)$. 
Then $- \essinf_{g \in {\cal C}^{\infty}_{st}}  \rho_{st}(g-X)$ belongs to $D_s$.
Then we have
\begin{equation*}
\bar\rho_{rs}(g_1- \underset{g_2 \in {\cal C}^{p}_{st}}{\essinf}\: \rho_{st}(g_{2}-X)) 
= \lim_{n\to\infty} \rho_{rs}(g_1-  \rho_{st}(h_{n}-X)),
\end{equation*}
from which we obtain
\begin{equation*}
\bar\rho_{rs}(g_1- \underset{g_2 \in {\cal C}^{p}_{st}}{\essinf}\: \rho_{st}(g_{2}-X)) = \underset{g_2 \in {\cal C}^{\infty}_{st}}{\essinf} \: \: \rho_{rs}(g_1- \rho_{st}(g_{2}-X).
\end{equation*}
So from equations \eqref{eqg1g2} and \eqref{eqstar1}, we get 
$$
\underset{g \in {\cal C}^{\infty}_{rt}}{\essinf}  \rho_{rt}(g-X) 
=
\underset{g_1 \in {\cal C}^{\infty}_{rs} }{\essinf} (\bar \rho_{rs}(g_1- 
\underset{g_2 \in {\cal C}^{\infty}_{st}}{\essinf}\: \rho_{st}(g_{2}-X))) \geq 
    \underset{g \in {\cal C}^{\infty}_{rt}}{\essinf}  \rho_{rt}(g-Y) .
$$
The result follows.
\end{proof}

\begin{proposition}
Fix $p \in [1,\infty]$. Let $(\rho_{st})_{s,t}$ be a strong time-consistent fully-dynamic risk measure on $(L_p({\cal F}_t))_t$. 
Then
\begin{equation}
x_{rt}(x_{st}(X))=x_{rt}(X),\qquad X \in L_{\infty}({\cal F}_t).
\label{eqxwtc}
\end{equation}
\label{propTC}
\end{proposition}
In particular this results show that, if $p=\infty$, the family $x_{st}$ is time-consistent, while
for $p<\infty$, the relationship is true only for essentially bounded claims.

\begin{proof}
For all $X \in L_{\infty}({\cal F}_t)$, the value $x_{st}(X)$ belongs to $L_{\infty}({\cal F}_s)$, see Lemma \ref{Linf} item 2. 
Then, as in \cite{AP}, $x_{rt}(x_{st}(X))=x_{rt}(X)$ follows from the time-consistency (Proposition \ref{propwtc}) applied with $X$ and 
 $Y=x_{st}(X)$.
\end{proof}

\begin{remark}
In the terminology of dynamic risk-measures, the property \eqref{eqxwtc} is sometimes called \emph{recursive}. See e.g. \cite{AP}.
This is not the same as the property of strong time-consistency. 
\end{remark}

\begin{remark}
Notice that when $p<\infty$ equation (\ref{eqxwtc}) cannot be proved in $L_p$.  Indeed it is not true in general that 
$x_{st}(X)$ belongs to $L_p(\F_s)$ for all $X \in L_p({\cal F}_t)$. 
This is the reason why we need to work with the extension of $x_{st}$ to $L^c_t$. Then we know that for every $X$ in $L^c_t$, 
$x_{st}(X)$ belongs to $L^c_s \subseteq L^c_t$.
\end{remark}

\begin{theorem}
\label{thmTCx}
Set $p\in [1,\infty)$.
Let $(\rho_{st})_{s,t}$ be a  strong time-consistent fully-dynamic risk measure on $(L_p(\F_t))_t$ such that $\rho_{0T}$ is 
dominated and sensitive. 
Let $x_{st}$ be the risk-indifference price on $L^c_t$ defined as in Definition \ref{defextx}. 
Then $(x_{st})_{s,t}$ on $(L_t^c)_t$ is time-consistent.

For every given time horizon $t \leq T$, the price system $(x_s)_{s}$, defined by restriction as  $x_s(X):=x_{st}(X)$, $s \leq t$, is 
time-consistent on 
the whole $L^c_t$.   Namely, for all $0 \leq r \leq s\leq t$, for all $X \in L^c_t$, we have $x_{s}(X) \in L^c_s$ and 
$$
x_{r}(X)=x_{r}(x_{s}(X)).
$$
\end{theorem}

\begin{proof}
Fix $t\in [0,T]$. 
We obtain that
$$
x_{rt}(x_{st}(X))=x_{rt}(X),\;\; \forall X \in L^c_t,
$$
by the density of $L_{\infty}({\cal F}_t)$ in $L^c_t$ and the uniform continuity for the c norm of $x_{st}$ for all 
$s \leq t$ 
(Proposition \ref{propxL}) and  (\ref{eqxwtc}). 
In turn, this gives the time-consistency of the family $(x_s)_{s}$, where $s\leq t$. 
\end{proof}

\bigskip
In the last part of this section we prove the regularity of the trajectories for the risk-indifferent price operators.

\begin{theorem}
 Fix some $t\in [0,T]$. Assume that for some $Q \in{\cal Q}$, $\gamma_{0t}(Q)=0$, where $\gamma_{0t}$ is the minimal penalty for $x_{0t}$. 
 Then for all $X \in L^c_t$, the stochastic process $x_{st}(X)$, $0 \leq s \leq t$, admits a c\`adl\`ag modification.
\end{theorem}

\begin{proof}
 Notice that $\gamma_{0t}(Q)=0$ implies that 
 \begin{equation}
 0=\sup_{X \in L_{\infty}({\cal F}_t)}(E_Q(X)-x_{0t}(X)).
 \label{eqSP}
 \end{equation}
 For all $r \leq s\leq t$, let $y_{rs}$ be the restriction of $x_{rt}$ to $L^c_s$. 
 From Theorem \ref{thmTCx} we deduce that the family $(y_{rs})_{0 \leq r \leq s\leq t}$ of operators on $(L^c_s)_s$ is strong time-consistent. 
 Its restriction to $(L_{\infty}(\F_s))_s$ is up to a minus sign a strong time-consistent normalised dynamic risk measure on $(L_{\infty}(\F_s))_s$.
Making use of (\ref{eqSP}), the proof of Lemma 3 and of Lemma 4 in Section 3.1 of  \cite{BN02}, applied with deterministic times, we can prove that $x_{st}(X)$ is the limit of $x_{s_nt}(X)$ in $L_1(Q)$, 
for every decreasing sequence $(s_n)_n$ converging to $s$ and for every $X \in L_{\infty}({\cal F}_t)$. 
We also get that $x_{rt}(X) \geq E_Q(x_{st}(X)|{\cal F}_r)$  for all $X \in L_{\infty}({\cal F}_t)$. 
Moreover, by the density of $L_{\infty}({\cal F}_t)$ in $L^c_t$ and  the uniform equicontinuity of $x_{st}$ for the $c$ norm (equation (\ref{eqceqcontx})), we can show that, for all $X \in L^c_t$, 
the value $x_{st}(X)$ is the limit of $x_{s_nt}(X)$ in $L_1(Q)$ and that $x_{st}(X)$ is a $Q$-supermartingale for all $X \in L^c_t$. 
Recall that the probability measure $Q \in {\cal Q}$ is equivalent with $P$,  then the modification theorem (see Theorem 4 page 76 in \cite{DM}) proves that, for all $X \in L^c_t$, the process $x_{st}(X)$ admits a c\`adl\`ag modification.
 \end{proof}

 \begin{remark} In view of the proposition above, our results are valid also with stopping times. Indeed, we could also have started our work with a fully-dynamic risk measure indexed by stopping times $(\rho_{\sigma,\tau})_{0 \leq \sigma \leq \tau \leq T}$ as in \cite{BN02} and obtain the same results of the present paper replacing deterministic times by stopping times.
We stress that our framework allows to give price evaluations to all American-type financial claims.
 \end{remark}
 

\bigskip
\section{Risk-indifference prices in $L_2$ and no-good-deal bounds}
\label{SecNGD}

Good-deal bounds were suggested by Cochrane and Saa Requejo \cite{CSR2000} in a static setting by fixing bounds on the Sharpe ratio. The idea is of identifying in this way those deals that are ``too good to be true''. Idea that was also considered by Bernardo and Ledoit \cite{BL2000} who are setting bounds on the gain-loss ratio and earlier Hodges \cite{Hodges1998} who uses a generalized Sharpe ratio derived from the negative exponential utility function, and also Cerny \cite{Cerny1999} where smooth utility functions are used to define good deals. See also \cite{Cerny-Hodges2002}.

In \cite{BNDN}, starting from \cite{CSR2000}, the relationship between bounds on the Sharpe ratio and no-good-deal pricing measures was detailed providing an equivalent definition of no-good-deal bounds expressed in terms of bounds on the Radon-Nykodim derivatives. 
Also it was possible to define the concept of dynamic no-good-deal bounds. 

In this section, in view of the nature of these concepts, it is natural to work with $p=2$.

We shall use dynamic no-good-deal bounds to provide a construction of risk-indifference prices in $L_2$. Indeed by the use of the bounds we can guarantee that, for all $s \leq t$, the risk-indifference price
$x_{st}$ (satisfying these bounds) is a well defined operator from $L_2(\F_t)$ to  $L_2(\F_s)$. 
We shall related this approach with the results of Section \ref{SecRI}. 
Moreover, our study provides a characterisation of the risk measures $(\rho_{st})_{s,t}$ so that the associated risk-indifferent prices are no-good-deal prices.

In the first part of this section we revise the fundamental concepts of no-good-deal bounds and provide some first results on the role of the bounds for the construction of convex operators in $L_2$.


\subsection{No-good-deal prices}

\begin{defn}
\label{no-good-deal measure}
A probability measure $Q\sim P$ is a \emph{no-good-deal pricing measure} if there are no good-deals of level $\delta>0$ under $Q$, that is, the Sharpe ratio is bounded:
\begin{equation}
\label{eqSR-2}
-\delta \leq \frac{E(X)-E_Q(X)}{\sqrt {\textrm{Var}(X)}} \leq \delta,
\end{equation}
for all $X\in L_2(\mcal{F}_T,P)\cap L_1(\mcal{F}_T,Q)$ such that $\textrm{Var}(X) \ne 0$.
Equivalently, we can say that $Q \sim P$ is a \emph{no-good-deal pricing measure} if 
$\frac{dQ}{dP} \in L_2({\cal F}_T)$ satisfies
\begin{equation}
E\Big[\big(\frac{dQ}{dP}-1\big)^2\Big] \leq \delta ^2.
\label{SR2}
\end{equation}
\end{defn}

From a dynamic perspective, it is suitable to work with the following set of probability measures.
\begin{defn}
\label{ngd-set}
Let $s \leq t$. Define the set $\mathfrak{Q}_{st}$ of probability measures on $\mcal{F}_t$ as
$$
\mathfrak{Q}_{st}  := \big\{Q \ll P : Q_{|{\cal F}_s}\hspace{-1mm}=\hspace{-1mm} P\: and \; 
\frac{dQ}{dP} \in {\cal D}_{st} \big\},
$$
where
$$
{\cal D}_{st} := \big\{1+h_{s,t} : \; h_{st} \in L_2(\F_t), \: E[h_{s,t}\vert \F_s] =0 \:  E\big[h_{st}^2|{\cal F}_s\big] 
\leq \delta_{st}^2\big\} . 
$$
Here the family of non-negative real numbers $\delta_{st}$, $s,t\in [0,T]: s\leq t$, satisfies the condition for all $r \leq s \leq t$:
\begin{equation}
(\delta_{rs}\delta_{st}+\delta_{rs}+\delta_{st})= \delta_{rt}
\label{eqComp}
\end{equation}
and $\delta_{st}\rightarrow 0$, $t\downarrow s$.
\end{defn}
Remark that we can connect the bounds on the Sharpe ratio \eqref{eqSR-2} with the one here above by choosing $\delta_{st}:=\delta^{t-s}-1$, for some $\delta >1$.

Then the following definition is given.
\begin{defn}
\label{dSR}
A probability measure $Q \sim P$ is a \emph{dynamic no-good-deal pricing measure} if 
$\frac{dQ}{dP} \in L_2({\cal F}_T)$ satisfies
\begin{equation}
E\Big[\big(\Big(\frac{dQ}{dP}\Big)_t\Big(\frac{dQ}{dP}\Big)_s^{-1}-1\big)^2\vert \mcal{F}_s\Big] \leq \delta_{st}^2,
\end{equation}
for every $s\leq t$ and constants $\delta_{st}>0$ satisfying \eqref{eqComp}.
Here $\big(\frac{dQ}{dP}\big)_t:= E\big[ \frac{dQ}{dP} \vert \mcal{F}_t \big]$.
\end{defn}
Corresponding to these bounds on the Radon-Nykodim derivatives, we can characterise the no-good-deal bounds on prices. 
\begin{defn}
\label{ngd-bounds-def}
The \emph{no-good-deal bounds} on prices are the sub-linear and super-linear operators here below:
\[
\begin{split}
M_{st}(X) &:= \underset{Q \in \mathfrak{Q}_{st}} \esssup \,  E_Q[X|{\cal F}_s], \quad X\in L_2(\mcal{F}_t),\\
m_{st}(X) &:=  \underset{Q \in \mathfrak{Q}_{st}} \essinf \, E_Q[X|{\cal F}_s], \quad X\in L_2(\mcal{F}_t),
\end{split}
\]
where $\mathfrak{Q}_{st}$ is given in Definition \ref{ngd-set}. 
\end{defn}
Clearly, $m_{st}(X) = - M_{st}(-X)$, for $X\in L_2(\F_t)$.
The properties of these operators are studied in Proposition 5.8 in \cite{BNDN}.

\bigskip
Hereafter we study the representation of a general convex price operator 
$x_{st}: L_\infty(\F_t) \longrightarrow L_\infty(\F_s)$ satisfying the no-good-deal bounds. 
This is a crucial result for the study of no-good-deal risk-indifference prices.
\begin{proposition}
\label{NGD-rep} 
Set $s \leq t$. Let $x_{st}: L_\infty(\F_t) \longrightarrow L_\infty(\F_s) $ be a convex price operator 
Definition \ref{convex-price} 
satisfying the no-good-deal bound: for all $X \in L_\infty(\F_t): X \geq 0$, 
\begin{equation}
\label{M-bound}
x_{st}(X) \leq M_{st}(X).
\end{equation}
Let $\gamma_{st}$ be the minimal penalty of $x_{st}$ on $L_\infty(\F_t)$.
Then, for any probability measure $Q \ll P$: $Q_{\vert \F_s} = P$ and $E[\gamma_{st}(Q)]<\infty$, we have that 
$Q\in \mathfrak{Q}_{st}$ and the following representation holds
\begin{equation}
\label{NGD-eq-rep}
x_{st}(X) =  \underset{Q  \in \mathfrak{Q}_{st}}\esssup \;\Big( E_Q (X\vert \F_s) - \gamma_{st}(Q) \Big), 
\quad X \in L_\infty(\F_t).
\end{equation}
Conversely, if \eqref{NGD-eq-rep} holds than \eqref{M-bound} is satisfied.
\end{proposition}

\begin{proof}
From Proposition 1 in \cite{BN02} we have that a convex price operator $x_{st}$ defined on $L_\infty(\F_t)$ admits representation in the form
\begin{equation}
\label{base-rep}
x_{st}(X) =  \underset{\substack{Q  \ll P: {\cal Q}_{\vert \F_s} = P\\  E[\gamma_{st}(Q)]<\infty }}
\esssup \Big( E_Q (X\vert \F_s) - \gamma_{st}(Q) \Big), \quad X \in L_\infty(\F_t).
\end{equation}
Consider the operator
$$
y_{st}(X) := x_{st}(X) - E (X\vert \F_s), \quad X\in L_\infty(\F_t).
$$
From \eqref{M-bound} and Definition \ref{ngd-set}, we have that, for $X \in L_\infty(\F_t)$: $X\geq 0$
\begin{equation}
\label{*}
y_{st}(X) \leq \underset{h \in {\cal D}_{st}} \esssup  E  \big( (h-1) X\vert \F_s \big) \leq \delta_{st} (E( X ^2\vert \F_s ) \big)^{1/2}.
\end{equation}
Hence, from \eqref{base-rep} and \eqref{*}, we obtain
\[
\begin{split}
\underset{\substack{Q  \ll P: {\cal Q}_{\vert \F_s} = P \\  
\frac{dQ}{dP}=k}}\esssup \Big( E ((k-1)  X\vert \F_s) - \gamma_{st}(Q) \Big) &= y_{st}(X)  
\leq \delta_{st} (E( X ^2\vert \F_s ) \big)^{1/2},
\end{split}
\]
for $X \in L_\infty(\F_t)$: $X\geq 0$.
Consider any $Q \ll P$ such that ${\cal Q}_{\vert \F_s} = P$, $\frac{dQ}{dP}=k$ and $E[\gamma_{st}(Q)]<\infty$.
Then for every $\lambda >0$ we have
$$
E ((k-1)  \lambda X\vert \F_s) - \gamma_{st}(Q) \leq \lambda \delta_{st} (E( X ^2\vert \F_s ) \big)^{1/2}.
$$
Thus 
$$
E ((k-1)   X\vert \F_s)  \leq  \delta_{st} (E( X ^2\vert \F_s ) \big)^{1/2} + \frac{1}{\lambda} \gamma_{st}(Q).
$$
Taking $\lambda \to \infty$ we obtain that
$ E ((k-1)   X\vert \F_s)  \leq  \delta_{st}  (E( X ^2\vert \F_s ) \big)^{1/2} $.
Hence, setting $h:= k-1$, we conclude that $E(h^2 \vert \F_s) \leq \delta_{st}^2$. That is $Q \in \mathfrak{Q}_{st}$.
Then $x_{s,t}$ admits representation \eqref{NGD-eq-rep}. The converse is immediate.
\end{proof}
We remark that from the representation \eqref{NGD-eq-rep} we can see that the bound \eqref{M-bound} is satisfied \emph{for all} $X \in L_\infty(\F_t)$.

\begin{corollary}
\label{cor-NGD}
Set $s\leq t$. 
Let $x_{st}: L_\infty(\F_t) \longrightarrow L_\infty(\F_s)$ be a convex price operator satisfying the no-good-deal bound \eqref{M-bound}.
Then $x_{st}$ is continuous in the $L_2$-norm and admits a unique extension
$$
x_{st} \, : L_2(\F_t) \longrightarrow L_2(\F_s).
$$
This extension admits the representation
\begin{equation}
\label{base-rep-2}
x_{st}(X) = \underset{Q\in \mathfrak{Q}_{st}}\esssup \Big( E_Q (X\vert \F_s) - \gamma_{st}(Q) \Big), 
\quad X \in L_2(\F_t),
\end{equation}
and satisfies the no-good-deal bounds for all $X\in L_2(\F_t)$:
\begin{equation}\label{sand}
m_{st}(X) \leq x_{st}(X) \leq M_{st}(X).
\end{equation}
\end{corollary}

\begin{proof}
We can see that $x_{st}$ is continuous in the $L_2$-norm from \eqref{*}. 
Hence the operator can be uniquely extended as a mapping from $L_2(\F_t)$ to $L_2(\F_s)$  by continuity.
Moreover, the right-hand side of representation \eqref{NGD-eq-rep} is continuous in the $L_2$-norm, hence also extendable by continuity as a mapping from $L_2(\F_t)$ to $L_2(\F_s)$. Thus representation
\eqref{base-rep-2} holds.
Moreover, for all $X\in L_2(\F_t)$, $m_{st}(X) \leq  -x_{st}(-X) \leq x_{st}(X) \leq M_{st}(X)$.
\end{proof}

\begin{remark}
As for the Sharpe ratio in the static case, the dynamic no-good-deal bounds are defined without reference to some specific price model. 
In the case where the dynamics of some basic assets are given the question of no-free-lunch (and thus no-arbitrage) has been studied in \cite{JBN_BA}. It follows 
from Theorem 5.1 therein that a dynamic convex price system satisfies the no-free-lunch condition if and only if it admits a dual representation with equivalent local martingale measures for the basic assets. \\ 
In general, the no-good-deal bounds are not included in the no-arbitrage bounds. No-arbitrage no-good-deal bounds can be constructed by replacing in the Definition \ref{ngd-bounds-def} the set ${\cal Q}_{st}$ by its intersection with the set of equivalent local martingale measures.
\end{remark}


\subsection{Risk-indifference prices with no-good-deal bounds}

In this section we aim at finding necessary and sufficient conditions on the fully-dynamic risk measures,  such that the associated risk-indifference price system satisfies the no-good-deal bounds.

\bigskip
We recall that the risk-indifference price system $(x_{st})_{s,t}$ on $(L_\infty(\F_t))_t$ is a convex price system according to Definition \ref{defCPS}.

\bigskip
In view of the setting $p=2$, we consider $(\rho_{st})_{s,t}$ to be a fully-dynamic risk measure on $(L_2({\cal F}_t))_t$
and $(x_{st})_{s,t}$ to be the family of operators defined as in \eqref{eqRIP}, from Definition \ref{defRIP}. 
For $s \leq t$, we define 
$$
\breve \rho_{st}(X):=\rho_{st}(X)-\rho_{st}(0), \quad X\in L_2(\F_t).
$$
Then the restriction of $\breve \rho_{st}$ to $L_{\infty}({\cal F}_t)$ takes values in $L_{\infty}({\cal F}_s)$ and 
 the risk-indifference price $\breve x_{st}$ associated to $\breve \rho_{st}(X)$ coincides with $x_{st}$.
 In the following $\breve \alpha_{st}$ denotes the minimal penalty associated to 
 $\breve \rho_{st}$ on $L_{\infty}({\cal F}_t)$.

\begin{theorem}
 \label{thmNGD}
 Let $(\rho_{st})_{s,t}$ be a fully-dynamic risk measure on $(L_2({\cal F}_t))_t$ with 
 $\breve \rho_{st}$ and $\breve\alpha_{st}$ as before. The following two groups of assertions A and B are equivalent for all $s\leq t$:
 \begin{enumerate}
 \item [A1.] $E(\underset{g \in \mathcal{C}^{\infty}_{st}}\essinf \rho_{st}(g)) \in \R$.
 \item [A2.] $x_{st}(X)=\underset{g \in \mathcal{C}^{\infty}_{st}}\essinf \rho_{st}(g-X)- \underset{g \in \mathcal{C}^{\infty}_{st}}\essinf \rho_{st}(g)$ 
 satisfies the no-good-deal bounds on $L_\infty({\cal F}_t)$.
 \item [B1.] For any probability measure $R \ll P$ on ${\cal F}_t$, $R_{|{\cal F}_s}=P$, such that 
 $E(\breve\alpha_{st}(R))<\infty$. Then either 
  \begin{itemize}
  \item $R \in \mathfrak{Q}_{st} \cap I_{st}$
  \item or $R \in I_{st}^c$
 \end{itemize}
 with 
 $$
 I_{st}=\Big\{R \ll P, \;R_{|{\cal F}_s}=P:\;E\Big[\underset{g \in \mathcal{C}^{\infty}_{st}}\esssup E_R (g|{\cal F}_s)\Big]<\infty\Big\}.
 $$
 \item [B2.] The set $\mathfrak{Q}_{st} \cap I_{st}$ is non empty and there exists  $R_0 \in \mathfrak{Q}_{st} \cap I_{st}$ 
 such that 
 $E(\breve\alpha_{st}(R_0))<\infty$.
 \end{enumerate}
\end{theorem}

\begin{proof}
First we assume that A holds and we prove assertions B.\\
B1)
Recall that $x_{st}(X) = \breve x_{st}(X)$, $X \in L_\infty(\F_t)$. 
Denote $\gamma_{st}$ the minimal penalty for the restriction of $x_{st}$ to $L_\infty({\cal F}_t)$. 
Observe that for $R \ll P, \;R_{|{\cal F}_s}=P$:
\begin{eqnarray}
\gamma_{st}(R) 
& = & \underset{X \in L_\infty({\cal F}_t)} \esssup[E_R(X|{\cal F}_s)-x_{st}(X)] \nonumber\\
& =  &\underset{X \in L_\infty({\cal F}_t)} \esssup [E_R(X|{\cal F}_s)- \underset{g \in \mathcal{C}^{\infty}_{st}}  
\essinf \breve\rho_{st}(g-X)]
+\underset{g \in \mathcal{C}^{\infty}_{st}} \essinf \breve\rho_{st}(g) \nonumber\\
& = & \underset{X \in L_\infty({\cal F}_t)}  \esssup  \underset{g \in \mathcal{C}^{\infty}_{st}} \esssup[E_R(X-g|{\cal F}_s)- 
\breve\rho_{st}(g-X)+E_R(g|{\cal F}_s)]\nonumber\\
&+ &\underset{g \in \mathcal{C}^{\infty}_{st}} \essinf \breve\rho_{st}(g) \nonumber\\
& =  &\underset{Y \in L_\infty({\cal F}_t)}  \esssup [E_R(-Y|{\cal F}_s)- \breve\rho_{st}(Y)]+ 
\underset{g \in \mathcal{C}^{\infty}_{st}} \esssup E_R(g|{\cal F}_s)
+ \underset{g \in \mathcal{C}^{\infty}_{st}} \essinf \breve\rho_{st}(g) \nonumber
\end{eqnarray}
Hence,
\begin{equation}
\gamma_{st}(R)=\breve\alpha_{st}(R)+\underset{g \in \mathcal{C}^{\infty}_{st}} \esssup E_R(g|{\cal F}_s)
+ \underset{g \in \mathcal{C}^{\infty}_{st}} \essinf \breve\rho_{st}(g) 
\label{eq*1}
\end{equation}
From \eqref{eq*1} and (A1), taking $R \ll P$ on ${\cal F}_t$, $R_{|{\cal F}_s}=P$ with $E(\breve\alpha_{st}(R))<\infty$,
we have that
\begin{equation}
E(\gamma_{st}(R))<\infty \Longleftrightarrow E[\underset{g \in \mathcal{C}^{\infty}_{st}} \esssup E_R(g|{\cal F}_s)]<\infty.
\label{eqgammaeq}
\end{equation}
So either 
$E[ \esssup_{g \in \mathcal{C}^{\infty}_{st}} E_R (g|{\cal F}_s)]=\infty$, 
which means that $R \in I_{st}^c$,
or 
$E[ \esssup_{g \in \mathcal{C}^{\infty}_{st}} E_R (g|{\cal F}_s)]<\infty$, and $R \in I_{st}$. In this case, from 
(\ref{eqgammaeq}), assumption A2 and 
Proposition \ref{NGD-rep}, we conclude that $R \in \mathfrak{Q}_{st}$. 

B2) 
From the representation \eqref{base-rep}, there exists a probability measure $R_0$ such that $E(\gamma_{st}(R_0))<\infty$. 
The statement follows from the same arguments as above.

\bigskip
Now the converse, assume that B holds and we prove assertions A.\\
A1) 
Take $R_0 \in \mathfrak{Q}_{st} \cap I_{st}$ such that $E(\breve\alpha_{st}(R_0))<\infty$. 
Then 
\begin{eqnarray}
E[\underset{g \in \mathcal{C}^{\infty}_{st}} \essinf \breve\rho_{st}(g)] 
&\geq & E\big[\underset{g \in \mathcal{C}^{\infty}_{st}} \essinf E_{R_0}(-g|{\cal F}_s)\big]-E(\breve\alpha_{st}(R_0)) \nonumber\\
&= & -E\big[\underset{g \in \mathcal{C}^{\infty}_{st}} \esssup E_{R_0}(g|{\cal F}_s)] \big]-E(\breve\alpha_{st}(R_0))>-\infty\nonumber
\end{eqnarray}
A2) 
Consider $R \ll P$ on ${\cal F}_t$, $R_{|{\cal F}_s}=P$, such that  $E(\gamma_{st}(R))<\infty$. 
Note that $\breve\alpha_{st}(R) \geq -\breve \rho_{st}(0)=0$, and 
$ \esssup_{g \in \mathcal{C}^{\infty}_{st}} E_R (g|{\cal F}_s)]\geq 0$. 
From (\ref{eq*1}), taking expectation, we can see that $E(\breve\alpha_{st}(R))<\infty$ and 
$E[\esssup_{g \in \mathcal{C}^{\infty}_{st}} E_R (g|{\cal F}_s)] <\infty$.
From assumption B1, we have $R \in \mathfrak{Q}_{st}$. 
Then we have 
\begin{equation}
x_{st}(X)= \underset{R \in \mathfrak{Q}_{st}} \esssup[E_R(X|{\cal F}_s)-\gamma_{st}(R)],\;\; X \in \; L_\infty({\cal F}_t)
\label{eqrepinf}
\end{equation}
and $x_{st}$ satisfies the no-good-deal bounds on $L_\infty({\cal F}_t)$.
\end{proof}

Moreover, we have the following result.
\begin{corollary}
Let $(\rho_{st})_{s,t}$ be a fully-dynamic risk measure on $(L_2({\cal F}_t))_t$ such that 
$E(\essinf_{g \in \mathcal{C}^{\infty}_{st}} \rho_{st} (g)) \in \R$.
Assume that the operator $x_{st}(X)$, $X\in L_\infty(\F_t)$, as in \eqref{eqRIP} satisfies the no-good-deal bounds on $L^{\infty}(\F_t)$. 
Then $x_{st}$ has a unique continuous extension 
$$
x^{ngd}_{st}: L_2({\cal F}_t) \longrightarrow L_2({\cal F}_s)
$$ 
satisfying the 
no-good-deal bounds on $L_2({\cal F}_t)$.

Furthermore, this extension admits the representation
\begin{equation*}
x^{ngd}_{st}(X) = \underset{Q\in \mathfrak{Q}_{st}\cap I_{st}}\esssup \Big( E_Q (X\vert \F_s) - \gamma_{st}(Q) \Big), 
\quad X \in L_2(\F_t).
\end{equation*}
\end{corollary}

\begin{proof}
The extension follows directly from Corollary \ref{cor-NGD}.
The representation follows from \eqref{base-rep-2} in Corollary \ref{cor-NGD}  and the arguments of the proof of Theorem \ref{thmNGD}.
\end{proof}

We remark that we do not know whether the extension $x^{ngd}_{st}$ is also a risk-indifferent price. 
We recall that in Section \ref{SecRI} we had assumed that the fully-dynamic risk measure $(\rho_{st})_{s,t}$ is strong time-consistent and $\rho_{0T}$ is dominated and sensitive. With this, starting from $x_{st}$ 
on $L_\infty(\F_t)$ we obtained a risk-indifferent extension:
$$
x^{ri}_{s,t}: L^c_t \longrightarrow L^c_s
$$
where $L_2(\F_t) \subseteq L^c_t$. See Definition \ref{defextx} and Theorem \ref{thmextrho}.
Now we study the relations between these two extensions.

\begin{theorem}
\label{last}
Let $(\rho_{st})_{s,t}$ be a strong time-consistent fully-dynamic convex risk measure such that hypotheses B in Theorem \ref{thmNGD} are satisfied. Assume 
$\rho_{0T}$ is dominated (for a constant $K >0$) and sensitive. 
Then
$$
x^{ri}_{st}(X) = x^{ndg}_{st}(X) \quad P-a.s., \quad X\in L_2(\F_t),
$$
and $x^{ri}_{st}(X) \in L_2(\F_s)$, $X\in L_2(\F_t)$.

Moreover, the following representation holds:
\begin{equation}
\label{rep-riskm}
\rho_{st}(X) = \underset{R \in (\mathfrak{Q}_{st}\cap B^K_{st}\cap I_{st}) \cup ( B^K_{st}\cap I^c_{st}) }
\esssup \Big( E_R(-X \vert \F_s) -  \alpha_{st}(R) \Big), 
\end{equation}
where
$$
B^K_{s,t} := \Big\{ R \ll P: \, R_{\vert \F_s} = P, \: \underset{Q \in {\cal Q}}\sup E_Q\big( E_R(X \vert \F_s) \big)
\leq K \Vert X \Vert_2, \, \forall X \in L_2(\F_t) \Big\},
$$
and
\begin{equation}
\label{rep-RI}
x_{st}^{ri}(X) = \underset{R \in \mathfrak{Q}_{st}\cap I_{st}} \esssup \Big( E_R(X \vert \F_s) - \breve \gamma_{st}(R) \Big).
\end{equation}
\end{theorem}

\begin{proof}
From Proposition \ref{propxL} item 1 and the definition of the semi-norm $c$, we have that
$$
c(\vert x^{ri}_{st}(X) - x^{ri}_{st}(Y)\vert ) \leq c(\vert X-Y \vert) \leq K \Vert X-Y \Vert_2,
$$
for all $X,Y \in L_2(\F_t)$.
Then $x^{ri}_{st}$ is continuous for the $L_2$-norm and we can conclude 
$x^{ri}_{st}(X) = x^{ndg}_{st}(X)$ $P-a.s.$ and then also $x^{ri}_{st}(X) \in L_2(\F_s)$, for all $ X\in L_2(\F_t)$.

The representation 
$$
\rho_{st}(X) = \underset{R \in {\cal P}_{st} }\esssup \Big( E_R(-X \vert \F_s) -  \alpha_{st}(R) \Big), \quad X\in L_2(\F_t),
$$
is derived from Proposition \ref{Prop2rep} item 2.
We observe that ${\cal Q} = \tilde {\cal P}_{0T}$ and that for any $Q\in {\cal Q}$ we have 
$Q_{\vert \F_s}\in \tilde{\cal P}_{0s}$, from  Lemma \ref{LemmaER} item 2.
Moreover, from item 1 in Lemma \ref{LemmaER} we have that any $R \in \tilde{\cal P}_{st}$ is $R = \bar R_{\vert \F_t}$, i.e. the restriction to $\F_t$ of some $\bar R \in \tilde{\cal P}_{sT}$ .

Now take any $R \in {\cal P}_{st} \subseteq \tilde{\cal P}_{st}$. From Proposition \ref{propcomp} we have that, 
for any $Q\in {\cal Q}$, there exists $S\in {\cal Q}$ such that
$$
E_Q\big( E_R(X\vert \F_s) \big) = E_S (X) \leq K \Vert X \Vert_2, \quad X \in L_2(\F_t),
$$
where the last inequality is justified by Proposition \ref{prop1}.
Then $R \in B^K_{s,t}$.
Moreover, we observe that the definition of minimal penalty implies that $\breve \alpha_{st}(R) \leq 
\alpha_{st}(R) + \rho_{st}(0)$
and then $E(\breve\alpha_{st}(R))<\infty$ comes from the definition of ${\cal P}_{st}$. This allows to use property B1 
from 
Theorem \ref{thmNGD} and we can conclude that $ R\in ((\mathfrak{Q}_{st}\cap I_{st}) \cup I^c_{st})\cap B^K_{st}$. 
Hence representation \eqref{rep-riskm} follows. 
\end{proof}

\section{Conclusions.}
In this paper we study risk-indifference pricing for financial claims from a dynamic point of view. Our setting is kept free from the particular choice of underlying price dynamics and we allow for the possibility of modelling changes of risk evaluations by the use of a fully-dynamic risk measure to derive the risk-indifference prices.

\vspace{2mm}
Our results show that, if the risk-indifference price operators $(x_{st})_{s,t}$ are in $(L_\infty(\F_t))_t$, then they are a convex price system.
However, $(L_\infty(\F_t))_t$ is sometimes too restrictive for financial modelling. But when working with prices in  $(L_p(\F_t))_t$ with $p<\infty$, then we see that the risk-indifference prices $(x_{st})_{s,t}$ do not possess all the properties, as single operators (see the Fatou property) and as a family (see the time-consistency). The core issue is related to the very structure of the domain of these operators.

\vspace{2mm}
Then, for $p<\infty$, we resolve to extend these price operators to a better suited framework $(L^c_t)_t$, which is described using capacities. In doing this operation it is needed to preserve as much as possible the meaning of risk-indifference. Then we first extend appropriately the whole family of underlying risk measures. We do this in a way that guarantees that the new extended fully-dynamic risk measure is also strong time-consistent. This, in turn, will guarantee the properties of the extended risk-indifference price operators.
In this new framework we see that the extended prices $(\tilde x_{st})_{s,t}$ have all the properties on the extended spaces $(L^c_t)_t$ of a convex price system.

\vspace{2mm}
From a modelling point of view it is however still desirable to have risk-indifference prices in $(L_p(\F_t))_t$.
Then we look into a more specific economic context and we study risk-indifference prices that are also no-good-deal prices.
Our results characterise those strong time-consistent fully-dynamic risk measures that provide no-good-deal risk-indifference prices. In this case, all the operators are framed in the $(L_2(\F_t))_t$ setting.

\bigskip
\bigskip
{\bf Acknowledgements:} We wish to thank Nicole El Karoui, Marco Frittelli, Marco Maggis, and Emanuela Rosazza-Gianin for the interesting discussions and their fruitful comments on this work. 
This research has benefit of the support of CMAP, Ecole Polytechnique and of the projects FINEWSTOCH (239019) and STORM (274410) of the Research Council of Norway (NFR).

\bibliographystyle{plain}

\end{document}